\documentclass[aap]{imsart}

\RequirePackage{amsthm,amsmath,amsfonts,amssymb}
\RequirePackage[numbers]{natbib}

\usepackage{float}
\usepackage{bbm}
\usepackage{xcolor}
\usepackage{xparse}
\usepackage{xstring}
\usepackage{enumitem}
\usepackage[normalem]{ulem}
\usepackage[mathscr]{euscript}

\usepackage{microtype}
\usepackage{orcidlink}
\usepackage{mhequ}

\usepackage{tcolorbox}
\usepackage{tikz}
\usepackage{tikz-cd} 
\usetikzlibrary{calc}
\usetikzlibrary{arrows}

\usepackage[colorlinks,linkcolor=blue,citecolor=blue,urlcolor=blue]{hyperref}
\usepackage[capitalise]{cleveref}

\usepackage{soul}

\makeatletter
\newcommand*{\breakhref}{}
\DeclareRobustCommand*{\breakhref}{%
  \begingroup
  \hyper@normalise\breakhref@
}
\newcommand*{\breakhref@}[2]{%
  \endgroup
  \SOUL@setup
  \def\SOUL@everyspace##1{%
    ##1%
    \dimen@=\fontdimen2\font
    \advance\dimen@ by \fontdimen3\font
    \cleaders\hbox to \dimen@{%
      \hss
      \href{#1}{\ \vphantom{#2}}%
      \hss
    }\hskip\dimen@\relax
    \hspace{-\fontdimen3\font
        plus \fontdimen3\font minus \fontdimen3\font}%
  }%
  \def\SOUL@everysyllable{%
    \href{#1}{\the\SOUL@syllable\vphantom{#2}}%
  }%
  \def\SOUL@everyhyphen{%
    \discretionary{-}{}{}%
  }%
  \def\SOUL@everyexhyphen##1{%
    \SOUL@setkern\SOUL@hyphkern
    \href{#1}{##1\vphantom{#2}}%
    \discretionary{}{}{%
      \SOUL@setkern\SOUL@charkern
    }%
  }%
  \SOUL@{#2}%
}
\makeatother

\startlocaldefs
\numberwithin{equation}{section}
\theoremstyle{plain}
\newtheorem{theorem}{Theorem}[section]
\newtheorem{lemma}[theorem]{Lemma}
\newtheorem{proposition}[theorem]{Proposition}

\theoremstyle{remark}
\newtheorem{definition}[theorem]{Definition}
\newtheorem{assumption}[theorem]{Assumption}

\newtheorem{remark}[theorem]{Remark}
\newtheorem{example}[theorem]{Example}

\newcommand{\eqdef}{\stackrel{\mbox{\tiny\rm def}}{=}}
\newcommand{\eqlaw}{\stackrel{\mbox{\tiny\rm law}}{=}}

\DeclareMathOperator{\supp}{supp}
\DeclareMathOperator{\diam}{diam}

\DeclareMathOperator{\var}{var}

\newcommand{\A}{\ensuremath{\mathcal{A}}}
\newcommand{\B}{\ensuremath{\mathcal{B}}}
\newcommand{\C}{\ensuremath{\mathcal{C}}}
\newcommand{\D}{\ensuremath{\mathcal{D}}}
\newcommand{\cE}{\ensuremath{\mathcal{E}}}

\renewcommand{\S}{\ensuremath{\mathscr{S}}}

\newcommand{\F}{\ensuremath{\mathcal{F}}}

\newcommand{\fL}{\mathfrak{L}}

\newcommand{\N}{\ensuremath{\mathbb{N}}}
\renewcommand{\P}{\ensuremath{\mathcal{P}}}
\newcommand{\R}{\ensuremath{\mathbb{R}}}

\newcommand{\1}{\ensuremath{\mathbf{1}}
}

\newcommand{\cV}{\mathcal{V}}

\def\ff#1#2{\textstyle{\frac{#1}{#2}}}
\let\f\frac
\def\epsilon{\varepsilon}
\def\le{\leq}
\def\Lip{\mathrm{Lip}}

\def\E{\mathbb E}

\newcommand{\norm}[1]{\ensuremath\left\|#1\right\|}

\renewcommand{\geq}{\geqslant}
\renewcommand{\leq}{\leqslant}
\def\CC{{\mathcal C}}

\def\${|\!|\!|}
\def\F{{\mathcal F}}
\def\<{\langle}
\def\>{\rangle}
\def\dom{\mathrm{Dom}}
\def\cov{\mathrm {cov}}
\def\ii{\mathrm{int}}
\def\mfK{\mathfrak{K}}
\newcommand{\vertiii}[1]{{\left\vert\kern-0.25ex\left\vert\kern-0.25ex\left\vert #1 \right\vert\kern-0.25ex\right\vert\kern-0.25ex\right\vert}}
\newcommand{\rom}[1]{(\textup{\uppercase\expandafter{\romannumeral#1}})}

\newcommand{\Cc}{\mathcal{C}_c} 
\newcommand{\Ccinf}{\mathcal{C}_c^\infty} 
\newcommand{\cH}{\mathcal{H}}
\newcommand{\bD}{\mathbb{D}} 
\newcommand{\cU}{\mathcal{U}} 
\newcommand{\nueff}{\nu_{\mathrm{eff}}} 
\renewcommand{\deg}{\textnormal{deg}} 
\newcommand{\tfrk}{\mathfrak{t}} 
\newcommand{\bbullet}{\textcolor{blue}{\bullet}} 

\let\eps\varepsilon
\let\d\partial

\makeatletter

\DeclareRobustCommand{\TitleEquation}[2]{\texorpdfstring{\StrLeft{\f@series}{1}[\@firstchar]$\if%
		b\@firstchar\boldsymbol{#1}\else#1\fi$}{#2}}

\makeatother
\definecolor{DarkGreen}{rgb}{0.0, 0.6, 0.0}

\def\dash{\leavevmode\unskip\kern0.18em--\penalty\exhyphenpenalty\kern0.18em}
\def\slash{\leavevmode\unskip\kern0.15em/\penalty\exhyphenpenalty\kern0.15em}

\endlocaldefs

\begin{document}

\begin{frontmatter}
\title{Fluctuations of stochastic PDEs with long-range correlations}
\runtitle{Fluctuations of SPDEs with long-range correlations}


\begin{aug}
\author[A]{\fnms{Luca}~\snm{Gerolla}\ead[label=e1]{luca.gerolla16@imperial.ac.uk}\orcidlink{0009-0007-4637-5130}},
\author[B]{\fnms{Martin}~\snm{Hairer}\ead[label=e2]{martin.hairer@epfl.ch}\orcidlink{0000-0002-2141-6561}}
\and
\author[B]{\fnms{Xue-Mei}~\snm{Li}\ead[label=e3]{xue-mei.li@epfl.ch}\orcidlink{0000-0003-1211-0250}}

\address[A]{Imperial College London\printead[presep={,\ }]{e1}}
\address[B]{EPFL \& Imperial College London\printead[presep={,\ }]{e2,e3}}

\end{aug}

\begin{abstract}
We study the large-scale dynamics of the solution to a nonlinear stochastic heat equation (SHE) in dimensions $d \geq 3$ with long-range dependence. This equation is driven by multiplicative Gaussian noise, which is white in time and coloured in space with non-integrable spatial covariance that decays at the rate of $|x|^{-\kappa}$ at infinity, where $\kappa \in (2, d)$. Inspired by recent studies on SHE and KPZ equations driven by noise with compactly supported spatial correlation, we demonstrate that the correlations persist in the large-scale limit.  The fluctuations of the diffusively scaled solution converge to the solution of a stochastic heat equation with additive noise whose correlation is the Riesz kernel of degree $-\kappa$. Moreover, the fluctuations converge as a distribution-valued process in the optimal H\"older topologies.
\end{abstract}





\end{frontmatter}

\section{Introduction}

We investigate the long-term statistical properties of the nonlinear stochastic heat equation (SHE) on $\R^d$ subject to multiplicative random forcing. Consider the equation 
\begin{equ}
	\label{eq:basic_mSHE}
	\partial_t u(t,x)=\ff 12 \Delta u (t,x)+\beta \sigma(u(t,x)) \xi(t,x), \qquad u_0 \equiv 1\;;
\end{equ}
where $\sigma:\R_+ \to \R$ is a real valued function that modulates the interaction between the solution $u(t,x)$ and the noise process $\xi$.  
The noise $\xi$ is centred Gaussian, white in time and coloured in space. Unlike most previous work on the
subject,  this article focuses on the case where the noise exhibits heavy-tailed correlations that are
in particular non-integrable at infinity. We restrict ourselves to the high-dimensional case, which in this context means  $d \geq 3$. 
While it is relatively straightforward (see the proof of Theorem~\ref{thm:stat_field} below) to show that \eqref{eq:basic_mSHE} admits a space-time stationary solution $Z$, the main purpose of this article is to study its large-scale fluctuations, leading to a central limit theorem for distribution-valued stochastic processes.

The stochastic heat equation, while significant in its own right \cite{molchanov1991ideas}, has attracted considerable interest due to its connection, via the Cole–Hopf transform \cite{Bertini-Giacomin97}, to the KPZ (Kardar–Parisi–Zhang) equation, a widely studied 
surface growth model \cite{kardar1986dynamic}. The solution of the linear SHE, where 
 $\sigma(u)=u$, represents the partition function of a system of directed polymers in a quenched random potential, thereby establishing a connection to statistical mechanics \cite{Bertini-Cancrini-95}.

The question of large-scale fluctuations is of particular interest. Recent studies \cite{gu18edwards, magnen2018scaling, dunlap2020_KPZ_fluctuations, cosco2020law, Lygkonis-Zygouras20, gu20_nlmSHE, gu_komorowski21_kpz_torus_fluct, dunlap2018homo, comets2018fluctuation} have shown that, under the condition that the correlations of the driving noise are either compactly supported or decay sufficiently fast, both the SHE and KPZ equations exhibit Gaussian fluctuations at large scales in the weak disorder regime, namely when the interaction strength $\beta$ is small. In this regime, the fluctuations converge to the Edwards–Wilkinson model, with an effective variance that depends on the spatial mollification of the noise (cf. Remark~\ref{rem:effective_var}).

Assuming the noise correlation decays at a polynomial rate of order $\kappa \in (2, d)$ with an asymptotic profile, we demonstrate that spatial correlation in the noise persists in the scaling limit. This leads to fluctuations governed by the stochastic heat equation with additive, spatially coloured noise $\dot W^\kappa$. Our results differ from the case of compactly supported noise in two further significant ways. First, the scaling factor is of order $\epsilon^{2-\kappa}$, indicating that the decay order of the correlations plays
the role of an ``effective dimension''.  Second, the formula for the effective variance of the limit  $\cU$ differs from that in the integrable correlation case. (This difference 
is expected, since applying that formula to our setting leads to a diverging quantity, 
see Remark~\ref{remark-1.8} for more details.)

The natural question whether our theorem could inform the fluctuation analysis of KPZ equation or be applicable to the critical decay rate $\kappa=2$ reveals 
substantial challenges. The Cole-Hopf transform $h=\log(u)$ is a singular mapping, complicating direct application. This issue is addressed in \cite{GHL24_KPZ}. 
Analogous questions in the critical decay case are the subject of ongoing research.  In the compactly supported scenario, such questions have already been satisfactorily addressed for the \textit{anisotropic} KPZ equation \cite{AKPZ2D1,AKPZ2D2,MR4499841,AKPZ2D3} and 
SHE\slash KPZ \cite{CSZ17-universality,caravenna2019moments_CSZ18a,CSZ20-KPZ-full-subcritical, gu2020gaussian_KPZ_2D,tao2022gaussian} in the critical dimension $d =2$.

Our results hold when the coupling constant $\beta$ is sufficiently small.
We note that the existence of the weak disorder regime and a transition at $\kappa = 2$ was already pointed out in a similar
context in \cite{Lacoin}. 
 Whilst one may expect a phase transition as in the compactly supported covariance case \cite{MSZ16_weak_strong_mshe, cosco2020law}, where the law of large numbers and fluctuations limit cease to hold for higher values of $\beta$,
it is at present unclear whether a similar behaviour is present in the long-range correlated setting.

 Finally we remark that there has been extensive research on stochastic heat equations. In particular we would like to highlight the  references \cite{Gaertner88, Mueller91, Krug-Spohn92, Carmona-Molchanov94,  Bertini-Cancrini-95, Bertini-Giacomin97, Carmona-Viens98, Bertini-Giacomin99,  Conus-Joseph-Khoshnevisan13, Huang-Nualart-Viitasaari-Zheng20, Kohatsu-Higa-Nualart21, Chen-Eisenberg24} which focus on the study of large scale (time and \slash or space) dynamics.

\subsection{Main Results}

We let $u$ be as in \eqref{eq:basic_mSHE} and, for $\eps > 0$, we write $u_\eps(t,x) = u(t/\eps^2, x/\eps)$, 
which  solves
	\begin{equ}\label{eq:rescaled_eqn_kappa}
		\partial_t u_\epsilon(t,x)=\ff 12 \Delta u_\epsilon(t,x)+\beta \epsilon^{\frac \kappa 2 -1}  \sigma(u_\epsilon(t,x)) \xi^\epsilon(t, x), \qquad u_\epsilon(0, x) \equiv 1\;,
	\end{equ}
	where $\xi^\epsilon (t,x) = \eps^{-\frac\kappa2-1}\xi(t/\eps^2, x/\eps)$ is a  Gaussian noise 
	that is white in time with spatial 
	covariance $\epsilon^{-\kappa} R(\frac{x}{\epsilon})$ and the product is again interpreted as an Itô integral. 	
Given a test function $g$,  we furthermore set
\begin{equ} \label{Xt}
	X_t^{\epsilon, g}=\epsilon^{1-\frac \kappa 2}\int_{\R^d}  \bigl( u_\epsilon(t,x)-\E u_\epsilon(t,x)\bigr)g(x) dx.
\end{equ}
Note that $\E u_\epsilon(t,x) = 1$ by our assumption on the initial condition, and it is  straightforward to see 
that a law of large numbers holds.

Throughout the article, we make the following assumptions on the noise. 
\begin{assumption}\label{assump-noise}
The noise $\xi$  in (\ref{eq:basic_mSHE}) is a mean zero space time Gaussian noise with $\E [ \, \xi_t(x) \xi_s(y)\,]=\delta(t-s) R(x-y)$ where $R: \R^d\to \R$ is smooth and positive definite.  There exists
	$\kappa \in (2,d)$ such that  $|R(x)| \lesssim (1 + |x|)^{-\kappa}$  and, for any $x \in \R^d \setminus \{0\}$,
	\begin{equ}\label{ass:R_conv}
		\epsilon^{-\kappa} R( \epsilon^{-1}x ) \to |x|^{-\kappa} \quad \text{as} \quad \epsilon \to 0.
	\end{equ}
\end{assumption}

\begin{example}
	The function $R(x) = (1+  |x|^2)^{-\kappa/2}$ satisfies the 
	above assumption.
\end{example}

From assumption \eqref{ass:R_conv}, it is immediate that $\xi^\epsilon$ converges in law to a spatially correlated noise with covariance $|x|^{-\kappa}$. In Theorem 2.10, we demonstrate that the solution of (\ref{eq:basic_mSHE}) converges uniformly as $t\to \infty$ to a stationary field $Z$. Our main result can then be formulated as follows.

\begin{theorem} \label{thm:basic-convergence}  
	Let $d\geq 3$,  $\sigma$ be Lipschitz continuous, 
	and suppose that $\xi$ satisfies Assumption~\ref{assump-noise}.	Then, there exists a $\beta_0>0$ such that for all $\beta < \beta_0$, and for any time indices 
	$0 < t_1 \leq \dots \leq t_n$, and for any
	smooth functions with compact support $\{g_i\}_{i =1, \dots, n}~\subset~\Ccinf(\R^d)$,  the following random variables converge, as $\epsilon\to 0$, in distribution: 	$$  (X_{t_1}^{\epsilon,g_1}, \dots, X_{t_n}^{\epsilon,g_n}) \Rightarrow \bigg(\int_{\R^{d}} \cU(t_1,x) g_1(x) \, dx, \dots , \int_{\R^{d}} \cU(t_n,x)  g_n(x) \, dx\bigg). $$
	Here, $\cU$ denotes  the solution to   the additive stochastic heat equation
	\begin{equ}[e:limitEW]
		\partial_t \cU(t,x)=\ff 12 \Delta \cU (t,x) + \beta \nueff \, \dot W^\kappa (t,x), \qquad \cU(0, \cdot) \equiv 0,
	\end{equ}
	with effective variance $\nueff^2 =  |\E [ \sigma (Z(0))]|^2$  and covariance of the Gaussian noise given by
	\begin{equ} \label{eq:EW_noise_covariance_defn}
		\E [\,\dot W^\kappa(t,x) \dot W^\kappa(s,y)\,]= \delta(t-s) |x-y|^{-\kappa}.
	\end{equ}
\end{theorem}

\begin{remark}
	At present  it is unclear what would happen with $\kappa \leq 2$ and in the edge case $\kappa =d$.
	In the latter case one would expect to see Gaussian fluctuations, but with a logarithmic correction in the scaling, similarly to what happens in dimension $2$. In the case $\kappa < 2$ on the other hand, one
	expects to see a situation similar to the case of white noise in dimension $1$ with some
	non-trivial scaling exponents. 
\end{remark}

\begin{remark}\label{rem:beta_restriction}
The restriction to small values $\beta<\beta_0$ is used for deriving uniform bounds on the solution $u$ and its Malliavin derivative. It eludes us whether there is a critical value of $\beta$
beyond which \eqref{eq:basic_mSHE} fails to admit a stationary solution. 
\end{remark}

In Section~\ref{sec:weak_convergence}, we also show that $X^\epsilon$ converges weakly to $\cU$ as a  H\"{o}lder continuous distribution-valued stochastic process.
\begin{theorem}\label{thm:weak_conv_1}
Let $d\geq 3$, $\sigma$ be Lipschitz continuous, 
	and suppose that $\xi$ satisfies Assumption~\ref{assump-noise}.	
Then, for any $\gamma \in (0,\frac 1 2)$ and $\alpha < 1 - \frac\kappa2 - 2\gamma$, there exists $\beta_{\gamma,\alpha} >0$ such that, for any $\beta <\beta_{\gamma,\alpha}$ and %
$T>0$, $X^\epsilon$ converges to $\cU$ weakly in $\C^\gamma ([0,T], \C^{\alpha})$.
\end{theorem}

\begin{remark}\label{rem:effective_var}
	By a simple reasoning, the law of large numbers $u_\eps \to \E u_\eps = 1$ holds, but this is in fact
	validated a posteriori by Theorem~\ref{thm:weak_conv_1}, see also Lemma~\ref{lem:Du_L2+cov_estimate} below. 
	\end{remark}
\begin{remark}	\label{remark-1.8}
	 Since $v_\eps = \epsilon ^{1-\kappa/2} ( u^\epsilon -1) $ solves the equation
	 $$\partial_t v_\epsilon(t,x)=\ff 12 \Delta v_\epsilon(t,x)+\beta  \sigma(\epsilon^{\kappa/2-1}v_\epsilon(t,x)+1) \xi^\epsilon(t, x), \qquad v_\epsilon(0, x) \equiv 0\;,$$
	one would naïvely expect that in the limit \eqref{e:limitEW},  $\nueff = \sigma(1)$, but this is incorrect due to resonances between the fluctuations of
	$u_\eps$ and $\xi^\eps$.	
	We now give a heuristic argument leading to the correct expression 
	$\nueff^2 = \lim_{t \to \infty} |\E [ \sigma (u(t,0))]|^2$.
	
	For this, we test the rescaled noise term in \eqref{eq:rescaled_eqn_kappa} against a test function $g$, thus considering the expression
	$$\int_{\R^d}  \int_0^t  \sigma(u_\epsilon(s,x))g(x)\; \xi^\epsilon(ds,dx),$$
	whose covariance is given by 
	
	\begin{equ}[e:covarSimple]
		\int_{\R^{2d}}  \int_0^t C(\eps^{-2}s, \eps^{-1}(x-y)) \,ds\, \epsilon^{-\kappa} R(\ff {x-y} \epsilon) g(x) g(y)\, dx\,dy\;,
	\end{equ}
	where we set
	\begin{equ}
		C(s,x) = \E [ \sigma(u(s,0)) \sigma(u(s,x))]\;.
	\end{equ}
	As in \cite{MSZ16_weak_strong_mshe}, we expect $u$ to converge in law to some stationary process $Z$ as $t\to \infty$, 
	so we expect $C(s,\cdot)$ to
	converge as $s\to \infty$ to some bounded function $\bar C: \R^d\to \R$ which, assuming that $Z$ is sufficiently mixing so the spatial correlation vanishes, furthermore has the asymptotic
	$\lim_{|x| \to \infty} \bar C(x) = |\E \sigma(Z)|^2$.

	In case $R$ is integrable so is $R \bar C$ and we expect \eqref{e:covarSimple}
	to converge (provided $\kappa = d$) to $\tilde \nu^2 \|g\|_{L^2}^2 t$ with
	$$\tilde \nu^2 = \int_{\R^d}\bar C(y) R(y)\, dy\;,$$
	thus leading to the limit equation $\d_t \cU=\ff 12 \Delta  \cU +\beta \tilde \nu  \xi$ where $\xi$ 
	is standard space time white noise.
	For $R$ with compact support, this heuristic agrees with the result  in \cite{MSZ16_weak_strong_mshe,gu18edwards}, see also \cite{cosco2020law}. 
	
	In our case $R$ is not integrable,  $\tilde \nu^2 =\infty$ for the SHE, but we have instead the large scale structure $\epsilon^{-\kappa} R(\frac{x-y}{\epsilon}) \approx |x-y|^{-\kappa}$. Since this is integrable near $0$, we expect \eqref{e:covarSimple} to behave like 
	\begin{equ}
		t \int_{\R^{2d}}  \bar C(\eps^{-1}(x-y)) \frac{ g(x) g(y)}{|x-y|^{\kappa}}\, dx\,dy
		\approx t |\E \sigma(Z)|^2 \int_{\R^{2d}}  \frac{ g(x) g(y)}{|x-y|^{\kappa}}\, dx\,dy\;,
	\end{equ}
	which is precisely what is suggested by Theorem~\ref{thm:weak_conv_1}.
\end{remark}

\begin{remark}
	The isotropic convergence assumption \eqref{ass:R_conv} can be relaxed to 
	\begin{equation*}
	\epsilon^{-\kappa} R( \epsilon^{-1}x ) \to \theta\big(\tfrac x {|x|}\big) |x|^{-\kappa} \quad \text{as} \quad \epsilon \to 0,
	\end{equation*}
	for some positive continuous function $\theta: \partial B(0,1) \to \R_+$. This results in an Edwards--Wilkinson limit driven by Gaussian white noise with covariance
	\begin{equation*}
	\E [\, \dot W^\kappa(t,x) \dot W^\kappa(s,y)\,]= \delta(t-s) \,\theta\big(\tfrac {x-y} {|x-y|}\big) |x-y|^{-\kappa},
	\end{equation*}
	in place of \eqref{eq:EW_noise_covariance_defn}.	
\end{remark}

\subsection{Proof and organisation of the paper}
Our basic strategy is to apply Stein's method, Proposition~\ref{prop:multi_Stein}, as in \cite{huang18-2020Stein_CLT_mshe}. For this, let $Z_i :=  \int_{\R^{d}} \cU (t_i, x) g_i(x) dx$, for a collection of test functions $g_i$ and times $t_i$ and $\cU$ the solution to the stochastic heat equation (\ref{e:limitEW}). Let $ \Sigma$ denote the covariance matrix of $\{Z_i\}$, and write $X^{\epsilon, g}$ as the divergence of $v^{\epsilon,g}$ through the mild formulation of the equation. It is sufficient to show that the quantity,
$\langle DX^{\epsilon, g_i}_{t_i}, v^{\epsilon,g_j}_{t_j} \rangle_{\cH} - \Sigma_{i,j} $, converges to zero in $L^2(\Omega)$.  We decompose this quantity into the sum  of the difference of the covariances between $X_{t_i}^{\epsilon, g_i}$ and $Z_j$ \dash which converges to zero by direct computation \dash and a
remainder term. The main technical hurdle is to  control the $L^2(\Omega)$ norm of this residual term,
which we bound by a multiple integral involving singular kernels. The latter are then shown to be bounded with an adaption of the general framework of \cite{hairer2016analyst_BPHZ}, see Lemma~\ref{lem:diagram_integration}.

The proof of Theorem~\ref{thm:weak_conv_1} is achieved by demonstrating the tightness of $X^\epsilon$ in $\C^\gamma ([0,T], \C^\alpha)$ for suitable $\gamma<\frac 1 2$ and $\alpha <-1$. This is shown by formulating a Kolmogorov type criterion (Proposition~\ref{prop:time-holder_distribution_process-app}) and leveraging uniform moment bounds. We review the spaces $\C^\alpha\subset \S^\prime$ in the beginning of Section~\ref{sec:weak_convergence}.

\medskip\noindent\emph{Organisation}. In Section~\ref{sec:prelims} we collect some basic Gaussian\slash covariance estimates, the key Feynman diagrams integration Lemma~\ref{lem:diagram_integration}, moment bounds, and decorrelation estimates on $u$. Section~\ref{sec:proof-thm1} is devoted to the proof of Theorem~\ref{thm:basic-convergence}, whereas the proof of tightness and Theorem~\ref{thm:weak_conv_1} are in Section~\ref{sec:weak_convergence}.

\subsubsection*{Notation and conventions}

\begin{enumerate}
\item  We assume $(\Omega, \F, \mathbb{P})$ is a complete probability space supporting the Gaussian noise $\xi$ and we denote by $\| \cdot \|_p$ the $L^p(\Omega)$ norm.
\item We write $\hat f(z) = \int_{\R^d} f(x) e^{-i z \cdot x}dx$ for the Fourier transform of $f$. 
\item Whenever $a\lesssim b$, we mean $a \leq C b$ for some constant $C>0$ independent of $\epsilon$.
\item For $a, b \in \R$, we denote their minimum with $a \wedge b = \min (a, b)$. 
\item Throughout the following sections (see Lemma \ref{lem:moments} and onwards) the coupling threshold $\beta_0$ and implicit inequalities constants $C$ will depend on $d, R$ and $\sigma$, we avoid keeping track of this to lighten notation.
\item We denote with $|\sigma|_{\Lip}$ the Lipschitz constant of the nonlinearity $\sigma$.
\item For $E\subset \R^d$ open, we write $\Ccinf(E)$ for the set of  smooth functions on $E$ with compact support, and we set $\D = \Ccinf(\R^d)$. 
\item For $r \in [0,1)$,
we denote by $\C^{r} = \C^{r}(\R^d, \R)$ the space of little-Hölder continuous functions, 
namely the $r$-Hölder functions $f$ such that $\lim_{y \to x} |y-x|^{-r}|f(y) - f(x)| = 0$ 
locally uniformly in $x$. This is
the completion of smooth functions under the system of seminorms
\begin{equ}
	\| f \|_{E, \C^r}  := \sup_{
		\substack{x, y \in E \\x\neq y }
	} \frac{|f(x) - f(y)|}{|x-y|^ r}\;,
\end{equ}
where $E \subset \R^d$ is open and bounded. Note that this space is separable, unlike the space
of all $r$-Hölder continuous functions.

We also consider the closure of smooth compactly supported function under weighted H\"older norms, e.g.\ take 
$w(x)=(1+|x |^2)^{-\f {p} 2}$, and $\|f\|_{C^{\alpha, w}}=\|wf\|_{C^\alpha}$. In this case the standard 
embedding theorem holds.
For $r \ge 1$, the space $\C^r$ is defined recursively as the space of continuously differentiable 
functions such that each directional derivative is in $\C^{r-1}$. We denote by $\C^r_c$ the subset 
of little-H\"{o}lder continuous compactly supported functions.
\end{enumerate}	

\section{Preliminary estimates}\label{sec:prelims}
Recall that the product on the right-hand side of the multiplicative stochastic heat equation \eqref{eq:basic_mSHE} is interpreted as an It\^o integral.
Given the flat initial condition $u(0,x) \equiv 1$, its (mild) solution is given by the solution to the fixed point problem
\begin{equ} \label{eq:mild_soln}
	u(t,x) = 1 + \beta \int_{0}^{t} \int_{ \R^d} P_{t-s}(x-y) \sigma(u(s,y)) \xi(ds, dy), 
\end{equ}
where $P_{t}(x) = (2\pi t)^{-d/2} e^{-\frac{|x|^2}{2t}}$ denotes the $d$-dimensional heat kernel. 
The  well-posedness theory for \eqref{eq:mild_soln} 
follows from classical theories \cite{walsh1986introduction, daPrato_Inf}.

From Assumption~\ref{assump-noise}, since $\hat R >0$ and, by Plancherel's identity,
\begin{equ}
	\int_{\R^d} \hat{R}(z) |z|^{-2} dz = c \int_{\R^d} R(y) |y|^{-d+2} dy \lesssim \int_{\R^d} (1+ |y|)^{-\kappa} |y|^{-d+2} dy < \infty\;,
\end{equ}
the spectral measure $\hat{R}$ satisfies ``Dalang's condition''
$	\int_{\R^d} \frac{\hat{R}(z)}{1+|z|^2} dz <\infty$.
Together with the condition that $\sigma$ is Lipschitz continuous and $R$ is bounded, this implies that the solution $u$ is function valued, \cite[Thm~13, Remark~14]{dalang1999extending}. See also \cite[Thm~0.2]{peszat2000nonlinear}.

\subsection{Elementary Gaussian and covariance estimates}

In this subsection we gather the main estimates involving the Gaussian kernel and covariance $R$ used below. First, we have the following pointwise estimates.

\begin{lemma}\label{lem:gaussian_time_bounds}
	Given any $\lambda \in (0,  d)$, the following bounds hold uniformly in $s >0$, and  $x \in \R^d\setminus\{0\}$:
	\begin{equ}
		P_{s}(x)  \lesssim s^{-
			\frac \lambda 2} |x|^{\lambda - d},  \qquad
		|\d_s P_{s}(x)|  \lesssim  s^{-\frac \lambda 2-1} |x|^{\lambda - d}.        \label{eq:comp_est}
	\end{equ}
	Furthermore, for every $\gamma \in [0,1]$ and $t>s$, one has
	\begin{equ}[e:boundDiffHeatTime]
	|P_t(x) - P_s(x)| \lesssim (t-s)^{\gamma} s^{-\gamma -\f \lambda  2}|x|^{\lambda-d}\;.
	\end{equ}
\end{lemma}
\begin{proof}
For the first bound, it suffices to recall that, by scale invariance of the heat kernel, $|\d_t^k P_t(x)| \lesssim (\sqrt{|t|} + |x|)^{-d-2k}$.
For the second bound, it follows from \eqref{eq:comp_est} that 
\begin{equ}
|P_t(x)-P_s(x)| \le \int_s^t |\d_r P_{r}(x)|\,dr
\lesssim |x|^{\lambda-d} \int_s^t r^{-1-\f\lambda2}\,dr
\lesssim |x|^{\lambda-d} s^{-\f\lambda 2-1} \bigl(s \wedge |t-s| \bigr)\;,
\end{equ}
and \eqref{e:boundDiffHeatTime} follows at once.
\end{proof}

\begin{lemma}\label{kernel-R}
	If $\kappa \in (0, d)$, then for any $r > 0$, any $x \in \R^d$, 
	\begin{equs}
		\int_{\R^d} \f{P_{r} (z-x)} {|z|^\kappa}dz \lesssim (|x|+\sqrt r)^{-\kappa},  \qquad \int_{\R^d} \f{P_{r} (z-x)} {1+|z|^\kappa}dz \lesssim (1+|x|+\sqrt r)^{-\kappa}.
	\end{equs}
\end{lemma}
\begin{proof}
It is sufficient to show the first inequality, the second follows from the fact that $ \int_{\R^d} {P_{r} (z-x)} {(1+|z|^\kappa)^{-1}}dz\le 1$ and a bound of the order $\min(1,  (|x|+\sqrt r)^{-\kappa})$ is equivalent to a bound $ C(1+|x|+\sqrt r)^{-\kappa}$.

By a change of variable,
$$\int_{\R^d} \f{P_{r} (z-x)} {|z|^\kappa}dz=r^{-\f \kappa 2} \int \f{P_{1} (z-\f x {\sqrt r})} {|z|^\kappa}dz.$$
As in \cite[Sec.~6.3]{HQ}, we now define for $\alpha,\beta \ge 0$ the spaces $\B_{\alpha,\beta}$ of functions
$F$ such that
\begin{equ}
\|F\|_{\alpha,\beta} \eqdef \sup_{|x| \le 1}|x|^\alpha |F(x)| + \sup_{|x| \ge 1}|x|^\beta |F(x)| < \infty \;.
\end{equ}
Note that one has $P_1 \in \B_{0,d}$ and $z\mapsto |z|^{-\kappa} \in \B_{\kappa,\kappa}$.
It is shown in \cite[Lem.~6.8]{HQ} that, provided that $\alpha_i < d$ and $\beta_1 + \beta_2 > d$, one has
\begin{equ}
\|F_1 \star F_2\|_{\alpha,\beta} \lesssim \|F_1\|_{\alpha_1,\beta_1}\|F_2\|_{\alpha_2,\beta_2}\;,
\end{equ}
with $\alpha = 0 \vee (\alpha_1+\alpha_2-d)$ and $\beta = (\beta_1 + \beta_2-d)\wedge \beta_1\wedge \beta_2$.
(Note that the statement given there is on $\R^{d+1}$ with parabolic scaling. The exact same proof applies to the
Euclidean scaling.)
In particular, it shows that the convolution of the two functions of interest to us belongs to 
$\B_{0,\kappa}$, whence the bound follows.
\end{proof}
We will also  frequently use the following fact: for any numbers $\alpha, \beta > -d$ such that $\alpha + \beta + d < 0$, one has
\begin{equ}\label{eq:homog_conv}
	\int_{\R^d} |z|^\alpha |x -z|^\beta dz = c_{\alpha, \beta}|x|^{\alpha+\beta+d}\;,
\end{equ}
for some finite constant $c_{\alpha,\beta}$. (Simply perform the change of variables $z = |x| y$.)

We gather some often used estimate about $R$ in the following lemma. Given $\kappa \in (2, d)$, the kernel $| \cdot|^{-\kappa}$ is integrable near $0$ so that the following estimates hold.

\begin{lemma}\label{lem:gaussian_cov_est}
	Let $R:\R^d\to \R^d$ be a Borel measurable function such that $|R(x)| \lesssim \frac{1}{1 + |x|^\kappa}$. If $\kappa \in (2, d)$, there exists $C_R > 0$ such that for any $r, s > 0$, any $x \in \R^d$, we have the bounds
	\begin{equs} \label{eq:time_decay_est}
		\int_{\R^{2d}} P_{r}(z_1-x) P_s(z_2)\big| R(z_1 -z_2)\big| dz_1 dz_2  &\leq C_R (1+|x|+\sqrt{r+s})^{-\kappa}, \\ 		\int_0^\infty \int_{\R^{2d}} P_{r}(z_1-x) P_r(z_2) \;\big|R(z_1 -z_2)\big| dz_1 dz_2\; dr
		& \leq C_R ( 1 \wedge |x|^{2-\kappa} ) \label{eq:control_space}\\		
		\int_{\R^d} \frac {|R(z)|} {|z-x|^{d-2}} dz  &\leq C_R(1 \wedge |x|^{2-\kappa})\;.\label{eq:control_R}
	\end{equs}
\end{lemma}
\begin{proof}
	Through a change of variable and convolution of Gaussian kernels,
	\begin{equs}
		{} &\int_{\R^{2d}} P_{r}(z_1-x) P_s(z_2) \big|R(z_1 -z_2)\big| dz_1 dz_2  =   \int_{\R^d} P_{r+s} (z-x) \big|R(z)\big|dz\\
		& \lesssim   \int_{\R^d} \f{P_{r+s} (z-x) }{1+|z|^k} dz \lesssim  (1+|x|+\sqrt {r+s})^{-\kappa}\;,
	\end{equs}
	where we used Lemma~\ref{kernel-R} in the last inequality. The bound \eqref{eq:control_space} follows
	by integrating over $r$.
	The last bound follows from \eqref{eq:homog_conv}, noting that the integral is also bounded, 
	uniformly over $|x| \le 1$.
\end{proof}

In the case of $R$ bounded and $L^1$-integrable (e.g.\ $\kappa >d$ or compactly supported $R$) the estimates \eqref{eq:time_decay_est}-\eqref{eq:control_space} hold with $\kappa$ replaced by $d$. In the following subsection we introduce a framework to handle systematically integration of multiple kernels and their associated small \slash large scale behaviours. 

\subsection{Feynman diagrams and integration lemma}\label{sec:Feynman}

Throughout this section, we fix a family $\{K_\tfrk: \tfrk \in  \fL\}$,
where $\mathfrak L$ denotes some finite index set and each $K_\tfrk$ is a function
$\R^d \to \R$ (also sometimes referred to as a ``kernel''). All our kernels are assumed
to be smooth, except at the origin where they are allowed to have a singularity.

Consider now a directed graph $\Gamma = (\cV,\cE)$ where
$\cV$ is a finite index set (interpreted as the vertices of $\Gamma$) and
$\cE \subset \cV \times \cV$ (with elements denoted 
by $e = (e_-,e_+)$). Given $\tfrk\colon \cE \to \fL$,
it will be convenient to have bounds on general integrals of the form 
\begin{equ} \label{eq:I_Gamma}
	I_\Gamma(\varphi) = \int_{(\R^{d})^{\cV}}\Bigl( \prod_{e \in \cE} K_{\tfrk(e)} (x_{e_+} - x_{e_-}) \Bigr) \varphi(x_{\cV_\ell})\, dx\;.
\end{equ}
Here, $\cV_\ell \subset \cV$ is some fixed subset of vertices (sometimes called ``legged vertices'', see
\cite{hairer2016analyst_BPHZ} for more details on terminology in a slightly more general setting),
 $x_{\cV_\ell} \in (\R^{d})^{\cV_\ell}$ denotes the restriction of $x \in (\R^{d})^{\cV}$ to $\cV_\ell$,
  and  $\varphi \in \C_c((\R^{d})^{\cV_\ell})$ is a test function.
We also sometimes write $\cV_{\ii} = \cV \setminus \cV_\ell$ for the ``interior vertices''.
We also think of $I_\Gamma$ as a distribution on $(\R^{d})^{\cV_\ell}$, in which case
integration takes place over the interior vertices and the legged vertices are kept fixed.
A triple $(\cV,\cE,\tfrk)$ is also called a ``Feynman diagram''.
In our applications, we indicate edge labels by colour \slash degree in the diagram and we ignore 
the orientation of edges corresponding to symmetric kernels.

In order to bound \eqref{eq:I_Gamma}, we need some quantitative control on the kernels $K_\tfrk$.
To this end, we assume that there are ``degree'' maps 
$$\deg_0 \colon {\mathfrak L} \to (-\infty,0), \qquad  
\deg_\infty \colon  {\mathfrak L} \to(-\infty,0),$$
capturing the small \slash large scale behaviour of the corresponding kernel:
\begin{equ} \label{eq:deg_ineq}
	| K_\tfrk (x)| \, \1_{|x|\leq 1}  \leq c_1  |x|^{\deg_0  \tfrk}, \qquad 
	| K_\tfrk (x)| \, \1_{|x| > 1}  \leq c_2  |x|^{\deg_\infty  \tfrk}. 
\end{equ}

We will also need the following two notions in order to formulate the assumptions guaranteeing
integrability of \eqref{eq:I_Gamma} at small \slash large scales respectively.
\begin{enumerate}
\item A \emph{subgraph} $\bar{\Gamma} \subset \Gamma$ consists of a subset of edges $\bar{\cE} \subset \cE$, as well as the subset of 
vertices $\bar{\cV} $  incident to at least one edge in $\bar{\cE}$.
Note that a subgraph does not need to contain all the edges in $\Gamma$ that are connected vertices in $\bar{\cV}$.
\item Let $\P_\Gamma = \{A_i\}$ be a partition of $\cV$.  An edge $(e_-, e_+)$ is said to be a ``traversing'' edge, if the two vertices belong to different elements of the partition. 
A partition  $\P_\Gamma$ with at least two elements is said to be \emph{tight} if there exists $A \in \P_\Gamma$ with  $\cV_\ell \subset A$.
\end{enumerate}

Combining \cite[Prop.~2.3 \& 4.1]{hairer2016analyst_BPHZ}, see also \cite{weinberg1960} for the
a formulation for the small-scale behaviour, we have the following lemma:
\begin{lemma}\label{lem:diagram_integration}
Let $\Gamma$ be a Feynman diagram with at least one edge.	Suppose that the following small and large scale conditions are satisfied.
	\begin{enumerate}
	\item For every non-empty subgraph $\bar \Gamma=(\bar \cV, \bar \cE)$,
	\begin{equ}\label{eq:small_scales_crit}
		\deg_0 \,\bar \Gamma := \sum_{e \in \bar \cE} \deg_0 \tfrk(e) + d ( |\bar \cV| - 1) \,\, > \,\,0. 
	\end{equ}
	\item For every tight partition $\P_\Gamma$, denoting $ \cE(\P_\Gamma)$ its set of traversing edges, 
	\begin{equ}\label{eq:large_scales_crit}
		\deg_\infty \, \P_\Gamma := \sum_{e \in \cE(\P_\Gamma)}  \deg_\infty \tfrk(e) + d ( |\P_\Gamma| - 1) \,\,< \,\, 0.
	\end{equ}
	\end{enumerate}	Then  $I_{\Gamma} (\varphi)<\infty$  for every $\varphi \in \Cc((\R^d)^{\cV_\ell})$.
\end{lemma}

\begin{remark}\label{remark:tree_diagram}
	If $\Gamma$ is a tree (no undirected cycle), 
	then \eqref{eq:small_scales_crit} is equivalent to imposing that
	$\deg_0 \mathfrak t (e) > -d$ for every
	$e \in \cE$.
\end{remark}

It is appropriate to state here an example, which will be used in the next section.
\begin{example}\label{lem:diagram0_finite}
Let $f,g \in \Cc(\R^d)$. Then for any $\kappa \in(0,d)$  and any $\lambda \in (0, \frac \kappa 2)$, 
	\begin{equ} \label{I-Gamma0}
		I_{\Gamma_0}(f,g; \lambda) :=  \int_{\R^{4d}} |y_1 -y_2|^{-\kappa} \Bigl( \prod_{i=1}^2 |x_i -y_i|^{\lambda-d}\Bigr)
		|f(x_1) g(x_2)| dy dx<\infty, 
	\end{equ}
	where $dy = dy_1 dy_2$ and $dx = dx_1 dx_2$. 
\end{example}
Since $\kappa<d$ and $\lambda>0$,  as \eqref{eq:homog_conv}, $\int |x_2-y_2|^{\lambda-d}|y_1-y_2|^{-\kappa} dy_2= C|y_1-x_2|^{\lambda-\kappa}$. Therefore, integrate out $y_2$ in $I_{\Gamma_0}$, it reduces to a constant multiple of the integral $I_{\Gamma_0'}$.  Since $2 \lambda< \kappa$, we repeat the procedure, eliminating $y_1$, we arrive at the estimate:
\begin{equs}
	I_{\Gamma_0} \lesssim I_{\Gamma_0'} &= \int_{\R^{3d}}  |x_1 -y_1|^{\lambda-d} |x_2 -y_1|^{\lambda-\kappa}  
	|f(x_1) g(x_2)| dy_1 dx \\
	&\lesssim \int_{\R^{2d}} |x_1-x_2|^{2 \lambda -\kappa} |f(x_1) g(x_2)| dx <\infty.
\end{equs}
This straightforward calculation provides an excellent opportunity to demonstrate graphical interpretations of fractional kernel integrals and the applications of Lemma  \ref{lem:diagram_integration}. 
Integrating out  $y_2$ corresponds to collapsing the edge $[y_1, y_2]$ in $\Gamma_0$ and removing the node $\{\bbullet_{y_2}\}$ from the diagram.
The  Feynman diagram labelled $\Gamma_0$ is associated with the integral.  Once integrated out all possible edges between internal vertices $\cV_\ii$,  the resulting integral is then associated with the reduced graph labelled $\Gamma_0'$,
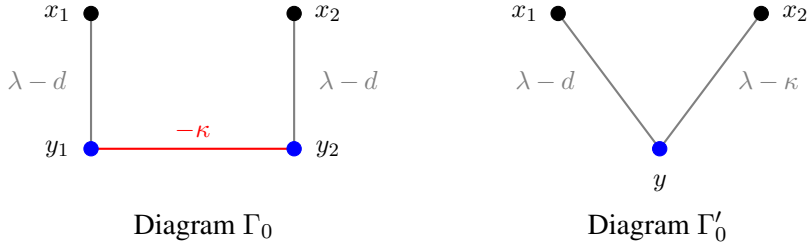
\begin{center}
	\hspace{-5em}
	\begin{tikzpicture}[scale=0.9, every node/.style={transform shape}]
	\node at (-0.5,0) {$y_1$};
	\node at (3.5,0) {$y_2$};
	\node at (-0.5,2) {$x_1$};
	\node at (3.5,2) {$x_2$};
	\node at (-0.8, 1) { \textcolor{gray}{$\lambda -d$}};
	\node at (3.8 , 1) { \textcolor{gray}{$\lambda -d$}};
	\draw[color = red, thick] (0,0) -- (3,0) node [midway, above] {$-\kappa$};
	\draw[color = gray, thick] (0,0) -- (0,2);
	\draw[color = gray, thick] (3,0) -- (3,2);
	
	\draw[blue, fill=blue] (0,0) circle (3pt);
	\draw[blue, fill=blue] (3,0) circle (3pt);
	\draw[fill=black] (0,2) circle (3pt);
	\draw[fill=black] (3,2) circle (3pt);

	\hskip 40pt
	\node at (5.0,2) {$x_1$};
	\node at (9,2) {$x_2$};
	\node at (7,-0.5) {$y$};
	\node at (5.3, 1) { \textcolor{gray}{$\lambda -d$}};
	\node at (8.6 , 1) { \textcolor{gray}{$\lambda -\kappa$}};
	\draw[color = gray, thick] (6.9,2) -- (8.4,0);
	\draw[color = gray, thick] (9.9,2) -- (8.4,0);
	\draw[black, fill=black] (6.9,2) circle (3pt);
	\draw[black, fill=black] (9.9,2) circle (3pt);
	\draw[blue, fill=blue] (8.4,0) circle (3pt);
	
	%
	%
	\end{tikzpicture}
	\\
	\hspace{-1.5em}Diagram $\Gamma_0$   
	\hskip 116pt Diagram $\Gamma_0'$  \hskip 30pt 
	\\
\end{center}

The small scales $\deg_0$ and the large scales $\deg_\infty$ of each edge are equal and labelled. To explain the terminology we work with $\Gamma_0$. Note that all 3 labels are greater than $-d$, then the small scale criterion  is trivially satisfied by Remark~\ref{remark:tree_diagram}.
Let $ \cV_\ell = \{ \bullet_{x_1}, \bullet_{x_2}\}$  denote the set of single-legged vertices and the remaining $\cV_\ii =\{ \bbullet_{y_1}, \bbullet_{y_2} \}$. A tight non-trivial partition must have an element $A_0$ containing $\cV_\ell$, so they are:
$\P_1= \{ \bullet_{x_1}, \bullet_{x_2},\bbullet_{y_1} \}\cup\{\bbullet_{y_2}\} $, $\P_2= \{ \bullet_{x_1}, \bullet_{x_2},\bbullet_{y_2} \}\cup\{\bbullet_{y_1}\} $,  $\P_3:= \{ \bullet_{x_1}, \bullet_{x_2}\} \cup \{\bbullet_{y_1},\bbullet_{y_2}\} $, and $ \P_4:=\{ \bullet_{x_1}, \bullet_{x_2}\} \cup \{\bbullet_{y_1}\}\cup\{\bbullet_{y_2}\} $. By the Lemma the integral is finite if  $\deg_\infty (\P_i)<0$ for each, which is trivial to check. 
The traversing edges for $\P_1$ are: $[y_1, y_2]$ and $[x_2, y_2]$, hence $\deg_\infty (\P_1)=\lambda-d -\kappa+d<0$ in this case, and similarly for $\P_2$. Whilst, $\deg_\infty \,\P_3  =  2(\lambda-d) + d = 2\lambda -d$, and $\deg_\infty \, \P_4=   2(\lambda-d) -\kappa + 2d = 2\lambda - \kappa <0$.

\subsection{Moments and large time asymptotics of SHE}
In this subsection,  we  prove some key moment estimates and limit results  for proving Theorem~\ref{thm:basic-convergence} in the next section, extending those from \cite{gu20_nlmSHE}. 
We also prove a spatial covariance estimate (Lemma~\ref{lem:Du_L2+cov_estimate}) using the Clark--Ocone formula. 

\smallskip\noindent \textbf{Notation.} We view $\xi$ as an isonormal Gaussian process on the Hilbert space $\mathcal{H}$,  the closure\footnote{In fact, $R$ is in principle allowed to be 
	such that there
	are non-trivial elements with zero norm, in which case we also quotient them out.} of $\Ccinf(\R \times \R^d)$ under the inner product 
$$  \langle f , g  \rangle_\cH = \int_{-\infty}^{\infty} \int_{\R^{2d}} f(s, y_1) g(s, y_2) R(y_1-y_2) dy_1 dy_2 ds.
$$
Thus, $\E [  \xi(f) \xi(g)] =  \langle f , g  \rangle_\cH $ for $f, g \in \cH$. 
By  \eqref{eq:control_space} in Lemma~\ref{lem:gaussian_cov_est}, the heat kernel belongs to $\cH$, 
provided  $|R(x)| \leq \frac{c_R}{1 + |x|^\kappa}$ with $\kappa \in (2, d)$.
\smallskip

We consider the space $\bD^{1,2} \subset L^2(\Omega)$ of real valued random variables $X$ 
such that  
\begin{equ}
	\| X \|^2_{\bD^{1,2}} :=  
	\E \big( |X|^2 \big)
	+ \E \big( \| DX\|_\cH^2  \big) < \infty, 
\end{equ}
where $D$ is the Malliavin derivative operator. 
We often identify $DX$ with a function-valued representative (among its equivalence class in $\cH$), writing 
$D_{r,z}X$ for its pointwise evaluation at $(r , z)$ in $ \R \times \R^d$. Note that, for $h \in \cH$, 
one has $D_{r,z}  \xi(h) = h(r,z)$. 
We write $\delta: \bD^{1,2}(\cH) \to L^2(\Omega)$ for the adjoint of $D$, also referred to as divergence
operator or Skorokhod integral.
If $w \in \dom \, \delta$ is $\F_t$-adapted, the Skorokhod integral $\delta(w)$ coincides with the It\^o integral 
$\int_{\R \times \R^d} w(s,y) \xi(ds,dy)$.

By \cite[Prop.~1.3.8]{nualart2006malliavin}, sufficiently regular elements $v \in \dom \,  \delta$ 
satisfy the commutation relation
\begin{equ}\label{eq:commutation_D_adjoint}
	D_{r,z} \delta(v) = v(r,z) + \int_{-\infty}^\infty \int_{\R^{d}} D_{r,z} v(s,y) \xi(ds,dy).
\end{equ}
The integral on the right-hand side is in the It\^o sense as long as $D_{r,z} v(s,y)$ is $\F_s$-adapted (in general it is a Skorokhod integral). The equality holds as elements of $\cH$.

Combining the commutation relation \eqref{eq:commutation_D_adjoint} with the mild formulation of the
solution $u$ allows for a useful characterisation of its Malliavin derivative \cite[p.~5]{gu20_nlmSHE} 
and to derive the estimates below, see \eqref{eq:moment_Du}. 
The identity \eqref{eq:commutation_D_adjoint} will also be used in the next section to express the Malliavin derivative of the fluctuations $X^{\epsilon,g}_t$. Given the Lipschitz assumption on $\sigma$, we have $D_{r,z}\sigma(u(s,y)) = \Sigma(s,y) D_{r,z} u(s,y)$ for $\Sigma$ a bounded and adapted process (if $\sigma$ is continuously differentiable then $\Sigma(s,y) = \sigma'(u(s,y))$).

The results in \Cref{lem:moments}  and \Cref{thm:stat_field}
were shown in \cite[Lem.~2.1--2.2, Thm 1.1]{gu20_nlmSHE} for the case when $R$ is compactly supported.
The pointwise bounds on the Malliavin derivative
rely on the following estimate, which is inspired by  \cite[Lemma 2.2]{chen2019comparison} but 
drops the requirement for $R$ to be positive.
Write $\B_+( \R_+ \times \R^d)$ for the set of measurable functions from $\R_+ \times \R^d$ to $ \R_+$ and a map 
$J_t(s,x): \B_+( \R_+ \times \R^d) \times  \B_+( \R_+ \times \R^d) \to \R$ 
as follows
$$J_t(s,x)(g_1, g_2)=\int_{ \R^{2d}} 
P_{t-s}(x-y_1)P_{t-s}(x-y_2)| R(y_1-y_2)| g_1(s, y_1)g_2(s, y_2) dy_1dy_2.$$
\begin{lemma}\label{lem:ptw_criterion}
	Let $\lambda>0$ and $g_n: \R_+ \times \R^d \to \R_+$ be nonnegative measurable functions. 
	Assume the function $R: \R^d \to \R$ satisfies the following for some $\alpha >1$, 
	\begin{equ}\label{eq:cov_sing_ass}
		\int_{ \R^{2d}}  P_r(z_1 - x_1) P_r(z_2 - x_2) 
		\bigl| R(z_1 -z_2)\bigr| dz_1 dz_2 \leq C_R (1 + r)^{-\alpha}. 
	\end{equ}
	If $g_0(t,x) \leq P_t(x)$ and $g_n$ satisfy for all $n \geq 0$, $t > 0$ and $x\in \R^d$,
	\begin{equ}\label{eq:est_g_contraction}
		\begin{aligned}
			&g_{n+1}^2(t,x) - ( P_t(x))^2 \le 
			\lambda \int_0^t J_t(s,x)(g_n,g_n) ds <\infty,
		\end{aligned} 
	\end{equ}
	and if furthermore $\lambda \in (0, \ff 1{c_\alpha})$ where $c_\alpha := 4C_R \int_0^\infty (1 + s)^{-\alpha} ds$ then 
	\begin{equ}\label{eq:est_g_pointwise}
		\sup_{n \geq0} g_n(t,x) \leq \frac {P_t(x)}{\sqrt{1 - \lambda c_\alpha}} .
	\end{equ} 
	
\end{lemma}
\begin{proof} 
	Let $f_0(t,x) = P_t(x)$  and
	for $n \geq 0$, let $f_n$ be the positive solution to
	\begin{equ}[e:deffn]
		f_{n}^2 (t,x)  = P_t(x)^2 + \lambda  \int_0^t J_t(s,x)(f_{n-1}, f_{n-1})ds.
	\end{equ}
	We note that $g_n \leq f_n$ for all $n\geq 0$. Indeed $g_0\le f_0$ and, for all $n \ge 0$, $g_n\le f_n$ implies that 
	$$g_{n+1}^2(t,x) \le ( P_t(x))^2 + \lambda \int_0^t J_t(s,x)(f_n,f_n) ds=f_{n+1}^2(t,x)\;,$$
	so that $g_n \leq f_n$ since both functions are positive by assumption \slash definition.
	
	Letting $H_n= \sum_{j =0}^{n} (\lambda c_\alpha)^j $, we show  by induction that  $f_n(t,x) \leq P_t(x) \sqrt{H_n}$.
	The case $n = 0$ is trivial. Using the relation
	\begin{equ}
		P_{t-s}(x-y) P_s(y) = P_t(x) P_{\frac {s(t-s)} t}	(y - \tfrac s t x )
	\end{equ}
	and the induction hypothesis on $f_n$
	we have 
	\begin{equs}
		&\int_0^t J_t(s,x)(f_n,f_n) ds \\
		&\leq H_n
		\int_0^t \int_{ \R^{2d}} 
		P_{t-s}(x-y_1) P_{t-s}(x-y_2)| R(y_1-y_2)| P_s (y_1) P_s(y_2) dy ds\\
		&  \leq H_n P_t(x)^2
		\int_0^t \int_{ \R^{2d}} 
		P_{\frac {s(t-s)} t} (y_1 - \tfrac s t x )    P_{\frac {s(t-s)} t} (y_2 - \tfrac s t x )
		| R(y_1-y_2)| dy ds\\
		&  \leq C_R H_n P_t(x)^2 \int_0^t ( 1 +\tfrac {s(t-s)} t)^{-\alpha} ds 
		\leq 2 C_R H_n P_t(x)^2 \int_0^{t/2} (1 + \tfrac s 2)^{-\alpha} ds \\ 
		& \leq 4 C_R H_n P_t(x)^2 \int_0^\infty (1+ s)^{-\alpha} ds \le c_\alpha H_n P_t(x)^2\;, 
	\end{equs}
	where we used assumption \eqref{eq:cov_sing_ass}	 to obtain the penultimate line.
	It then follows from \eqref{e:deffn} that
	\begin{equ}
		f_{n+1}^2(t,x) = P_t(x)^2 +\lambda \int_0^t J_t(s,x)(f_n,f_n) ds \leq P_t(x)^2 (1+\lambda c_\alpha H_n) = P_t(x)^2 H_{n+1},
	\end{equ}
	concluding the induction step. It follows that $g_n(t,x) \leq  P_t(x) \sqrt{H_\infty} = P_t(x)/\sqrt{1 -  \lambda c_\alpha}$,	
	concluding the proof.
\end{proof}

Set $\|\sigma\|^2=  | \sigma(0)| ^2+  |\sigma|_{\Lip}^2$, 
then $\beta\|\sigma\|$ can play the role of $\beta$, by rescaling $\sigma$ so that $\|\sigma\|=1$. 
Our estimates  hold  if $\beta < \f c {|\sigma|_\Lip}$, where $c$ is a constant depending on $p, \kappa$ and $R$.
\begin{lemma}\label{lem:moments} 
	Let $d\geq 3$, $\sigma$ be Lipschitz continuous, 
	and suppose that $\xi$ satisfies Assumption~\ref{assump-noise}. 
Then, for any $p \ge 2$, there exists positive constants $\beta_0(p)$ and $C_p$, 
	such that if $\beta \leq\beta_0(p)$, the solution of (\ref{eq:basic_mSHE}) has uniform $p$th moments:
	\begin{align}
		\sup_{s \geq 0,  x\in \R^d } \| u(s,x)\|_p 
		= \sup_{s\geq 0} \| u(s,0)\|_p \leq C_p. \label{eq:moment_u}
	\end{align}
	Furthermore, $u(t,x)$ is Malliavin differentiable
	with $\sup_{t\geq 0, x \in \R^d} \| u(t,x) \|^2_{\bD^{1,2}} < \infty$
	and, for all $x,z \in \R^d$ and $t > r > 0$,
	\begin{equ}
		\| D_{r,z}u(t,x)\|_{p} \lesssim C_{p} P_{t-r}(x-z) \label{eq:moment_Du}.
	\end{equ} 
\end{lemma}

\begin{remark}
	Given the flat initial condition $u(0, \cdot) \equiv 1$ and translation invariance of $\xi$, the process $u(t,\cdot)$ is stationary in space. 	The constant $C_p$ is monotone increasing in $p$.
The threshold $\beta_0(p)$ is monotone decreasing in $p$. 
\end{remark}

\begin{proof}
The key for the proof is Lemma~\ref{lem:ptw_criterion}.  The proof is otherwise quite standard, see e.g. \cite[Lem.~2.1-2.2]{gu20_nlmSHE} and \cite[Thm 6.4]{chen2019spatial}. 

First we recall that by the Burkholder--Davis--Gundy inequality, for any time interval $I$,
\begin{equ}\label{BDG}
	\Big\|\int_{I\times \R^d} f(s,y) \xi(ds, dy)\Big\|_{p}
	\leq c_p \,\E \left[ \Big(\int_I\int_{\R^{2d}} f(s,y_1)f(s,y_2) R(y_1-y_2) dy ds\Big)^{\f p2 } \right]^{\f 1 p}.
\end{equ}
We apply it to the stochastic integral $ \int_0^r \int_{\R^{d}}P_{t-s} (x-y)  \sigma(u(s, y))\xi(ds, dy)$ in the mild formulation, with $0\leq r \leq t$, obtaining the bound below where we denote $dy = dy_1 dy_2$,
\begin{equ} \label{eq:BDG_est1}
	\| u(t,0)\|_p \leq 1 + c_p \beta \E  \biggl[  \biggl(\int_0^t \int_{\R^{2d}} \Bigl(\prod_{j=1}^2 P_{t-s} (y_j) \sigma(u(s, y_j)) \Bigr) R(y_1 -y_2) dy ds \biggr)^{\f p 2}  \biggr]^{\f 1 p}.
\end{equ}
Since $u(s, \cdot)$ is stationary, by H\"{o}lder's inequality,
\begin{equ}
	\|   \sigma(u(s, y_1))\sigma(u(s, y_2)) \|_{\f p 2} \leq \|  \sigma(u(s, 0)) \|_{p}^2 \le  2 | \sigma(0)| ^2+ 2 |\sigma|_{\Lip}^2 \|  u(s, 0) \|_{p}^2. 
\end{equ}
Then, letting $f(t) := \| u(t,0) \|_{p}$,  and  applying Minkowski inequality in \eqref{eq:BDG_est1},
\begin{equ}
	f(t) \leq   1 +  \sqrt{2} c_p \beta  \biggl(
	\int_{0}^{t}\int_{\R^{2d}}   P_{t-s} (y_1) P_{t-s}(y_2)  \big(  | \sigma(0)| ^2 +   |\sigma|_{\Lip}^2 f^2(s)\big) \big| R(y_1 -y_2)\big| dy ds \biggr)^{\f 1 2}. 
\end{equ}
By \eqref{eq:time_decay_est}, it follows that there exist
constants\footnote{The constants are $C_1 = c_p \sqrt{ 2C_R}  |\sigma|_\Lip$ and $C_2 =   1+ \beta c_p \sqrt{2C_R} |\sigma(0)|$, where $C_R$ is from \eqref{eq:time_decay_est}.  } $C_1, C_2 >0$ (independent of $t$) such that 
\begin{equ}
	f(t) \leq C_2 + \beta C_1   \biggl( \int_{0}^{t}
	 (1+ t-s)^{-\f \kappa 2 } 
	f^2(s) ds \biggr)^{\f 1 2}, 
\end{equ} 
Since $\kappa >2$,  the integral $c_\kappa = \int_{0}^{\infty} (1+ r)^{-\f \kappa 2 }   dr$ is finite. It follows that
$\sup_t f(t)\le C_2 +  \sqrt {c_\kappa} C_1 \beta \sup_t f(t)$. 
For any $\beta \le ( 2 \sqrt {c_\kappa} C_1)^{-1}$, we have that  $\sup_t f(t)\le 2 C_2$. The bound	\eqref{eq:moment_u} follows from the fact that the process $u(t,\cdot)$ is stationary in space.

For \eqref{eq:moment_Du}, let us consider $u_0(t,x) = 1$, and for $n\geq 0$ define the Picard iteration 
\begin{equ}
	u_{n+1}(t,x) = 1 + 
	\beta \int_{0}^{t} \int_{ \R^d} P_{t-s}(x-y) \sigma(u_n(s,y)) \xi(ds, dy).
\end{equ}
A similar argument as for $f(t)$ above shows that, for any $x\in \R^d$, 
\begin{equs}
	& \|u_{n+1}(t,x)-u_n(t,x)\|_p\\
	&\leq \beta c_p | \sigma|_{\Lip}\biggl( \int_{0}^{t}\int_{\R^{2d}}   P_{t-s} (y_1) P_{t-s}(y_2)  \|u_{n}(s,0)-u_{n-1}(s,0)\|_p^2 \big| R(y_1 -y_2)\big| dy ds \biggr)^{\f 1 2}\\
	& \leq  \f 12 \sup_s\sup_x \|u_{n}(s,x)-u_{n-1}(s,x)\|_p, 
\end{equs}
for sufficiently small $\beta$ as above.
Thus $\{u_n(t,x)\}$ is a Cauchy sequence and converges to the unique solution $u$. 

Moreover, it follows from \eqref{eq:commutation_D_adjoint} and \cite[Prop.~1.2.4]{nualart2006malliavin} that
\begin{equs}
D_{r,z} u_{n+1} (t,x) = 
{P_{t-r}(x-z)} \sigma(u_n(r,z)) 
+ \beta \int_{r}^{t} \int_{ \R^d} P_{t-s}(x-y) \Sigma_n D_{r,z} u_{n} (s,y) \xi(ds, dy), 
\end{equs}
where $\Sigma_n(t,x)$ is a stochastic process bounded by $|\sigma|_{\Lip}$, which, intuitively speaking, represents a version of the derivative  $\sigma'(u_n(s,y))$.
Observe that $\sup_t\sup_x\|\sigma(u_n(t,x))\|_p$ is uniformly bounded in $n$, c.f. \eqref{eq:moment_sigma_u}. 
We apply again BDG to the integral term, 
followed by Minkowski and H\"{o}lder inequalities, obtaining\footnote{The constants are given by $C_1' =2c_p^2 |\sigma|_\Lip^2$ and $C_2' = 4 \| \sigma\|^2 C_p^2$.}
\begin{equs}
\| D_{r,z} u_{n+1} (t,x) \|_p^2 &\leq C_2'  P^2_{t-r}(x-z)+ \\
&	\quad	\beta^2   C_1' \int_{r}^{t} \int_{ \R^{2d}} 
|R(y_1-y_2)| \Big( \prod_{i=1}^{2} P_{t-s}(x-y_i) 
\| D_{r,z} u_{n} (s, y_i) \|_p  \Big)
dy ds.
\end{equs}
The covariance bound  \eqref{eq:time_decay_est} implies  \eqref{eq:cov_sing_ass} 
with $\alpha = \kappa/2 > 1$, so that Lemma \ref{lem:ptw_criterion} applies\footnote{As long as $4 \beta^2 C_1' C_R c_\kappa < 1$, therefore can take $\beta_{0}(p)$ such that $ \beta_0^2 c_p^2 c_\kappa C_R  |\sigma|_\Lip^2 < \f 1 8$. }, 
in which we let $t = \theta + r$, $x = \eta +z$ and $g_n(\theta, \eta) = \norm{D_{r,z} u_n (\theta +r,\eta + z)}_p$, to obtain 
\begin{equ}
\| D_{r,z} u(t,x) \|_p \leq \sup_{n \geq0} 
\| D_{r,z} u_{n} (t,x) \|_p  
\lesssim C_p P_{t-r}(x-z)\;. 	
\end{equ}

To show that $u(t,x) \in \bD^{1,2}$, note that Cauchy--Schwarz and \eqref{eq:moment_Du} yield
\begin{equ}[e:covarDu]
\E [ D_{r, z_1} u(t,x) D_{r, z_2} u(t,x) ] \leq \prod_{i=1}^2 \|D_{r, z_i} u(t,x) \|_2 \leq C \, \prod_{i=1}^2 P_{t-r} (x-z_i).
\end{equ}
It then remains to note that \eqref{eq:control_space} and \eqref{e:covarDu} yield
$\| u(t,x) \|^2_{\bD^{1,2}} < \infty$, uniformly over $t$ and $x$, thus
concluding the proof.
\end{proof}

The next result is inspired by \cite[Thm.~1.1]{gu20_nlmSHE} and shows that the solution with constant initial condition
converges in law towards some stationary field $Z$ for large times, with an explicit convergence rate.
In fact, we show  the stronger convergence in probability in the pullback sense, so we write
$u^{(K)}(t, x)$ for the solution of the nonlinear stochastic heat equation with initial condition $1$ at time $-K$, namely
\begin{equ} [e:eqn_pullback]
u^{(K)}(t, x) = 1 + \beta \int_{-K}^{t} \int_{ \R^d} P_{t-s}(x-y) \sigma(u^{(K)}(s,y)) \xi(ds, dy).
\end{equ}
With these notations at hand, the claimed convergence is as follows.

\begin{theorem}\label{thm:stat_field}
Let $d\geq 3$, $\sigma$ be Lipschitz continuous, and suppose that $\xi$ satisfies Assumption~\ref{assump-noise}. Then, for any   $p>1$,  there exists a positive constant $\beta_1(p)$, such that for $\beta\leq\beta_1$, uniformly in $t >  - (K \wedge K')$:
\begin{equ} [eq:Lp_cvg_rate_pull_back]
\sup_{x \in \R^d}	\| u^{(K')}(t,x) - u^{(K)}(t,x) \|_p \lesssim \beta (1 + t + K\wedge K')^{\f{2 - \kappa}4}. 
\end{equ}
Furthermore, there exists a space-time stationary field $\vec { Z}$ such that, for
any compact region $\mfK \subset \R^{d}$ and every $t \in \R$, one has
$\lim_{K \to \infty} \E \sup_{x \in \mfK} |  u^{(K)}(t,x)-\vec {Z}(t,x) |^p \to 0$.
\end{theorem}

\begin{proof}
Without loss of generality, in the proof we assume that $0\le K<K'$.
Set
\begin{equ}[e:defMp]
	M_p := \sup_{x \in \R^d}\sup_{t > -K} (t+K+1)^{\f{\kappa -2}4} \| u^{(K')}(t,x) - u^{(K)}(t,x) \|_p\;.
\end{equ} 
For  $t > -K$, we have:
\begin{equs}
	u^{(K')}(t,x) - u^{(K)}(t,x) 
	& =\beta \int_{-K}^{t} \int_{ \R^d} P_{t-s}(x-y) [\sigma(u^{(K')}(s,y) ) - \sigma(u^{(K)}(s,y) )] \xi(ds, dy)\\
	& \quad +  \beta \int_{-K'}^{-K} \int_{ \R^d} P_{t-s}(x-y) \sigma(u^{(K')}(s,y) ) \xi(ds, dy). 
\end{equs}
Thus, by the Burkholder--Davis--Gundy inequality (\ref{BDG}), for a universal constant $c_p$,
\begin{equs}
	&\|u^{(K')}(t,x) - u^{(K)}(t,x) \|_{p}\\
	&\leq c_p\beta  \Big\|  \int_{-K}^t\int_{\R^{2d}} \Bigl( \prod_{j=1}^2 P_{t-s} (x-y_j) [\sigma(u^{(K')}(s,y_j) ) - \sigma(u^{(K)}(s,y_j) )] \Bigr) R(y_1 -y_2) dy ds   \Big\|_{\frac p 2}^{\f 1 2}\\
	&  \quad + c_p\beta\Big\|  \int_{-K'}^{-K}\int_{\R^{2d}} \Bigl( \prod_{j=1}^2 P_{t-s} (x-y_j)  \sigma(u^{(K')}(s,y_j) ) \Bigr) R(y_1 -y_2) dy ds  \Big\|_{\frac p 2}^{\f 1 2}.
\end{equs}
Denoting the two terms of the type $\|\ldots\|_{\f p2}^{\f12}$ as $A_1$ and $A_2$ respectively, 
we rewrite this bound as
\begin{equs}\label{uk}
	\|u^{(K')}(t,x) - u^{(K)}(t,x) \|_{p} \leq c_p\beta( A_1(t,x)+A_2(t,x)).
\end{equs}
By the Minkowski and H\"{o}lder inequalities, using \eqref{eq:time_decay_est} in the last step, we have 
\begin{equs}
	A_1(t,x)&\leq \Big( \int_{-K}^t\int_{\R^{2d}}  \Bigl(\prod_{j=1}^2 P_{t-s} (y_j-x) \| \sigma(u^{(K')}(s,y_j) ) - \sigma(u^{(K)}(s,y_j) )\|_p \Bigr) R(y_1 -y_2) dy ds \Big)^{\frac 1 2}
	\\
	&\leq |\sigma |_{\Lip}   \sup_{s>-K} (s+K+1)^{\f{\kappa -2}4}\sup_{y}\| u^{(K')}(s,y)  - u^{(K)}(s,y) \|_p  \times\\
	&\qquad  \Big( \int_{-K}^t (s+K+1)^{-\f{\kappa -2}2}\ \int_{\R^{2d}} P_{t-s} (y_1)  P_{t-s} (y_2) R(y_1 -y_2) dy ds \Big)^{\frac 1 2}   	\\
	&\leq |\sigma |_{\Lip} \,M_p  \Bigl( C_R \int_{-K}^t  (s+K+1)^{-\f{\kappa -2}2}(1+t-s)^{-\f\kappa2}  ds \Bigr)^{\frac 1 2}.
\end{equs}
 We split the integration region $[-K, t]$ into two equal parts.
On the interval $[-K,\f{ t-K} 2]$, we have
$t-s\ge \f{t+K}{2}$, so that
\begin{equ}
	\int_{-K}^{\f{ t-K} 2} (s+K+1)^{-\f{\kappa -2}2} (1 + t -s )^{-\frac \kappa 2 } ds
	\leq 2^{\f \kappa 2 -1}  (1+t+K)^{-\frac \kappa 2 } (t+K)\le 2^{\f \kappa 2 -1}  (t+K+1)^{1- \f \kappa 2}\;,
\end{equ}
while on $[\f{ t-K} 2, t]$ we have $s+K \ge \f{t+K}2$, so that
\begin{equ}
	\int_{\f{ t-K} 2}^{t} (s+K+1)^{-\f{\kappa -2}2} (1 + t -s )^{-\frac \kappa 2 } ds
	\le 2^{\f \kappa 2 -1}   (1+t+K)^{1-\f{\kappa}2}\int_{-\infty}^{t} (1 + t -s )^{-\frac \kappa 2 } ds\;.
\end{equ}
Combining both bounds, we obtain\footnote{The implicit constant is given by $ 2^{\f {\kappa -2} 4} (1+ \sqrt {c_\kappa}) |\sigma |_{\Lip} \sqrt {C_R}$, where $c_\kappa = \int_{0}^{\infty} [1 \wedge r^{-\f \kappa 2 }]  dr$ as before.}
\begin{equ}
	A_1(t,x)\lesssim   \; M_p\; (1 + t + K)^{\f{2 - \kappa}4} \;.
\end{equ}
Similarly, given that $\sup_{s, y,K'}\| \sigma(u^{(K')}(s, y)) \|_p$ is uniformly bounded  for $\beta<\beta_0(p)$ 
by a constant $c({\beta_0})$ (cf. \eqref{eq:moment_sigma_u} below), the second term is controlled up to a constant by 
\begin{equ}
	A_2 (t,x)\lesssim c({\beta_0}) \Big( \int_{-K'}^{-K}  (1 + t -s )^{-\frac \kappa 2 } ds \Big)^{\frac 1 2}
	\lesssim
	c({\beta_0}) (1 + t + K)^{\f{2 - \kappa}4}. 
\end{equ} 
Returning to (\ref{uk}) with these estimates, rearranging $(1 + t + K)^{\f{2 - \kappa}4}$, and taking suprema in $x$ in $t > -K$ of both sides, we obtain that for some constants $C_1$ and $C_2$\footnote{Where this time $C_1 = c_p 2^{\f {\kappa -2} 4}  (1+ \sqrt {c_\kappa}) |\sigma |_{\Lip} \sqrt {C_R}$, thus can fix $\beta_1 \le \f  1 {2C_1}$.  }, 
\begin{equs}
	M_p \leq \beta C_2 + \beta C_1  M_p. 
\end{equs}
For sufficiently small $\beta$, 
we have $M_p \leq \f{\beta C_2}{1-\beta C_1}\le 2 \beta C_2$. When combined with \eqref{e:defMp}, this
implies (\ref{eq:Lp_cvg_rate_pull_back}). This shows in particular that there exists a space-time stationary field $\vec Z(t,x)$ such that, for every $(t,x)$, 
$	\vec Z(t,x) = \lim_{K \to \infty} u^{(K)}(t, x)$ in probability.

To show that this $L^p$ convergence is locally uniform in $x$, we make use of \cite[Prop.~3.12]{MR4499013}: combined with Kolmogorov's criterion and a simple interpolation argument, it suffices to show that there exists some $\delta > 0$ such that, for every $p \ge 1$,
\begin{equ}
	\|u^{(K)}(t, x_1)-u^{(K)}(t, x_2)\|_{p} \lesssim |x_1-x_2|^\delta\;,
\end{equ}
uniformly over $K$, $t$, and $|x_1-x_2| \le 1$. Since our bounds are uniform,
we can restrict ourselves to the case $K=0$. Letting $G_{x_1, x_2}(s, y) = P_{s} (x_1 - y) - P_{s} (x_2 - y)$,  we have 
\begin{equs}
	\|u(t, x_1)-u(t, x_2)\|_{p} & = \Big\|  \int_{0}^{t} \int_{\R^{d}} G_{x_1, x_2}(t-s, y) \sigma(u(s, y)) \xi(ds, dy) \Big\|_p  \\
	& \lesssim \Big( \int_{0}^t\int_{\R^{2d}} \Bigl(\prod_{i=1}^{2} G_{x_1, x_2}(t-s, y_i) \| \sigma(u(s, y_i)) \|_p\Bigr)   R(y_1 - y_2) dy ds \Big)^{\f 1 2}.
\end{equs}
By \cite[Lem.~3.1]{chen2019comparison}, for  $\delta \in (0,1)$, we have 
\begin{equ}
	G_{x_1, x_2}(t-s, y) \lesssim (t-s)^{-\f \delta 2 } ( P_{2(t-s)} (x_1 - y) +  P_{2(t-s)} (x_2 - y) ) |x_1 -x_2|^{\delta}. 
\end{equ}
Then, since $\| \sigma(u(s, y)) \|_p$ is uniformly bounded,  we obtain from \eqref{eq:time_decay_est} the bound
\begin{equs}
	&\|u(t, x_1)-u(t, x_2)\|_{p} \\
	& \quad \lesssim 
	|x_1 -x_2|^{\delta}  \sum_{i, j = 1}^2 \Big( \int_{0}^t (t-s)^{- \delta }  \int_{\R^{2d}}  P_{2(t-s)} (x_i - y_1) P_{2(t-s)} (x_j - y_2)  R(y_1 -y_2) dy ds\Big)^{\f 1 2} \\
	& \quad \lesssim  |x_1 -x_2|^{\delta}  \Big( \int_{0}^t (t-s)^{- \delta }  (1 \wedge (t-s)^{- \f \kappa 2}) ds \Big)^{\f 1 2} \lesssim |x_1 -x_2|^{\delta}\;,
\end{equs}
as desired.
\end{proof}

\begin{remark}
We remark that the restrictions required on $\beta_0$ and $\beta_1$ are such that $\beta_1 = a_\kappa \beta_0$, for some constant $a_\kappa > 0$. Both are inversely proportional to $|\sigma|_\Lip$. 
\end{remark}
\begin{remark}
We expect to get the same limiting stationary process also if we start with an
initial condition $\phi$ that merely averages to $1$ in the sense that
$\lim_{t\to\infty} P_t \phi = 1$ in a suitable sense.
\end{remark}

We now proceed to obtain decorrelation estimates. With the moment bounds of Lemma~\ref{lem:moments} in place, we can show that $u(t,x) \in \mathbb{D}^{1,2}$. Moreover, using the Clark--Ocone formula and \eqref{eq:control_space}, we show that solutions decorrelate at rate $|x|^{2-\kappa}$ for large $x$.

\begin{lemma}\label{lem:Du_L2+cov_estimate} 
	Let $d\geq 3$,  $\sigma$ be Lipschitz continuous, and suppose that $\xi$ satisfies Assumption~\ref{assump-noise}.
	Let $f: \R\to \R$ be a Lipschitz function. For $ \beta< \beta_0(p)$,   
	\begin{equ}	\label{eq:moment_sigma_u}
		\sup_{s \geq 0, x \in \R^d} \| f (u(s,x))\|_p  \lesssim 1.
	\end{equ}
	For any $ \beta< \beta_0(2)$,  the following holds uniformly in  $x \in \R^d$,
	\begin{equ}
		\sup_{s>0} \cov\left( f\left(  u(s, x) \right) , f\left(  u(s, 0) \right)  \right) \leq C |f|_{\Lip}^2 \left( 1 \wedge |x|^{2 -\kappa}\right).
		\label{eq:cov_space_estimate}	
	\end{equ}
\end{lemma}
\begin{proof}
	Recall that $u(t,x)$ has uniform $L^p$ bounds by Lemma~\ref{lem:moments}. 
	Noting that $\| f (u(s,x))\|_p \leq  |f(0)| +| f |_\Lip \|u(s,x)\|_p$, the first inequality follows. 
	For the second bound, we note  $D_{r,z}f(u(t,x)) = \Sigma(t,x)\, D_{r,z}u(t,x)$, for some $\Sigma(t,x)$ bounded by $|f|_{\Lip}$. 
	We apply the following Clark--Ocone formula
	(see eg. \cite[Prop 6.3]{chen2019spatial} for a proof): for any $X \in \bD^{1,2}$, 
	\begin{equ}\label{eq:Clark--Ocone}
		X - \E [X] = \int_{\R_+ \times \R^d} \E [ D_{r,z} X | \F_r] \xi(dr,dz)\;.
	\end{equ}
	Applying this to $f(u(t,x))$, we obtain
	\begin{equ}
		f(u(t,x)) - \E f(u(t,x)) = \int_{0}^{\infty}\int_{\R^d} \E[ \Sigma(t ,x) D_{r,z}u(t,x) | \F_r] \xi(dr,dz).
	\end{equ}
	By It\^o's isometry, followed by the Cauchy--Schwarz inequality and (conditional) Jensen's inequality,
	we have 
	\begin{equs}
		{} & \cov\left( f(u(t,x)), f(u(t,0))\right) = \\
		& \qquad  \int_{0}^{\infty}\int_{\R^d} \E \big(
		\E[ \Sigma(t,x) D_{r,z_1}u(t,x) | \F_r] \,
		\E[ \Sigma(t,0) D_{r,z_2}u(t,0) | \F_r]
		\big)  R(z_1 -z_2) dz_1 dz_2 dr\\
		& \qquad \leq |f|_{\Lip}^2 \int_{0}^{t}\int_{\R^d} 
		\| D_{r,z_2}u(t,x) \|_2\| D_{r,z_2}u(t,0) \|_2 \big|R(z_1 -z_2)\big| dz_1 dz_2 dr.
	\end{equs}
	Then \eqref{eq:cov_space_estimate} follows by using again the $L^2(\Omega)$ estimate \eqref{eq:moment_Du} on $Du$ and \eqref{eq:control_space},
	\begin{equs}
		\cov\left( f(u(t,x)), f(u(t,0))\right) &\lesssim |f|_{\Lip}^2
		\int_{0}^{t}\int_{\R^{2d}}  
		P_{t-r} (x-z_1) P_{t-r}(z_2) \big|R(z_1 -z_2)\big| dz_1 dz_2 dr \\
		&\lesssim  |f|_{\Lip}^2  (1 \wedge |x|^{2-\kappa}),  
	\end{equs}
	concluding the proof of the second statement.
\end{proof}

\begin{remark}\label{rem:decorrelation} 
\begin{enumerate}
\item 	
For the linear stochastic heat equation,  the Feynman--Kac formula yields
\begin{equ}
	\cov ( u_\epsilon(t,x), u_\epsilon(t,0) ) = \E_B \exp\Big({\frac{\beta^2}{2}\int_{0}^{\frac{2t}{\epsilon^2}} R(\tfrac{x}{\epsilon} + B_r)  dr}\Big) -1. 
\end{equ} 
Here $\E_B$ denotes expectation with respect to an independent Brownian motion $B$.
Using Portenko's lemma \cite{portenko}, 	
then their covariance converges to $0$  as long as the following  holds
\begin{equ}
	\E_B\biggl[  \int_{0}^{\frac{2t}{\epsilon^2}} R(\tfrac{x}{\epsilon} + B_r)  dr \biggr] \leq \E_B\left[  \int_{0}^{\infty} \big| R(\tfrac{x}{\epsilon} + B_r)\big|  dr \right] = \int_{ \R^d} \frac{|R(z)|}{| \tfrac{x}{\epsilon} -z|^{d-2}} \, dz \to 0.
\end{equ}
\item By \eqref{eq:cov_space_estimate}, we see that 
\begin{equ}
	\cov (u_\epsilon(t,x), u_\epsilon(t,0)) \lesssim \epsilon^{\kappa -2} |x|^{2 -\kappa}  \to 0. 
\end{equ}
Since $\E[u_\epsilon]=1$, this immediately implies the law of large numbers:
\begin{equ}\label{eq:LLN}
	\int_{\R^d} u_\epsilon(t,x) g(x) dx \xrightarrow{\,\,\mathbb{P} \,\,} \int_{\R^d} g(x) dx.
\end{equ}
\item The asymptotic stationary field also decorrelates, $\cov( Z(\tfrac x \epsilon), Z(0)) \approx~0$.
\end{enumerate}		
\end{remark}

\section{Proof of Theorem~\ref{thm:basic-convergence}}\label{sec:proof-thm1}

Throughout this section, we assume $\beta < \beta_0$, where $\beta_0 >0$ is taken sufficiently small so that Theorem~\ref{thm:stat_field} and Lemma~\ref{lem:Du_L2+cov_estimate} hold, and we can apply Lemma~\ref{lem:moments} in the arguments below to 
estimate moments.\footnote{Keeping track of the various constraints, 
 one finds that it suffices to set $\beta_0 = \beta_0(4) \wedge \beta_1(2)$ with $\beta_0(p)$ as in Lemma~\ref{lem:moments} 
	and $\beta_1(2)$ as in Theorem~\ref{thm:stat_field}.}
Our main tool is the following result \cite[Prop.~2.3]{huang18-2020Stein_CLT_mshe}:
\begin{proposition} \label{prop:multi_Stein}
	Let $X = (X_1, \dots , X_n)$ be a random vector such that $X_i \in \bD^{1,2}$ and $X_i = \delta(v_i)$ for $v_i \in \text{Dom} \, \delta$, $i=1, \dots n$. Let $Z$ be an $n$-dimensional centered Gaussian vector with covariance matrix $(\Sigma_{i,j})_{1 \leq i,j \leq n}$. Then, for any function $h: \R^n \to \R$ in $\C_b^2(\R^n)$, letting 
	$\|D^2 h\|_\infty = \max_{1\leq i,j\leq n} \sup_{x \in \R^n}| \partial_{x_i} \partial_{x_j}  h (x)|$, we have
	\begin{equ}
		| \E h(X)-\E h(Z)|
		\le \frac 1 2 \|D^2 h\|_\infty \; \sqrt{
			\sum_{i,j=1}^{n} \E [ ( \Sigma_{i,j} - \<DX_i, v_j\>_\cH   )^2]}.
	\end{equ}
\end{proposition}

\begin{proof}[Proof of Theorem~\ref{thm:basic-convergence}]
	We follow the general strategy of \cite{gu20_nlmSHE}.	
Since $\CC^2$ functions with compact support are weak convergence determining over $\R^n$,
it follows immediately that, in the context of Proposition~\ref{prop:multi_Stein}, the convergence
	\begin{equ}\label{eq:aim_multi_conv}
		\<DX_i^\epsilon, v_j^\epsilon\>_\cH \longrightarrow \Sigma_{i,j} \quad \text{in}\,\, L^2(\Omega)\;,
	\end{equ}	
for a sequence $X_i^\eps = \delta(v_i^\eps)$ implies that $X^\eps$ converges in law to $Z$. 

In order to prove Theorem~\ref{thm:basic-convergence}, we then apply this result to
$X_i^\eps := X_{t_i}^{\eps, g_i}$, defined by (\ref{Xt}).  We also set
\begin{equ}
	Z_i \eqdef Z_{t_i}(g_i) =  \int_{\R^{d}} \cU (t_i, x) g_i(x) dx\;,
\end{equ}
where $\cU$ is as in Theorem~\ref{thm:basic-convergence}. By the definition of the mild solution, 
we have $ X_i^\epsilon  = \delta(v_i^\epsilon)$ with 
\begin{equ}[e:defvi]
	v_i^\epsilon(s,y) := v_{t_i}^{\epsilon,g_i}(s,y) =   \beta \epsilon^{1-\kappa/2} \1_{[0,\frac{t_i}{\epsilon^2}]} (s) \sigma(u(s,y)) \int_{\R^{d}} P_{\frac{t_i}{\epsilon^2} -s } ( \tfrac x \epsilon
	- y) g_i(x) dx\;,
\end{equ}
so it remains to show that \eqref{eq:aim_multi_conv} holds with $\Sigma_{i,j}=\E (Z_iZ_j)$.

Applying the commutation relation \eqref{eq:commutation_D_adjoint} to $D\delta (v_i^\epsilon)$, we obtain the decomposition 
$\<DX_i^\epsilon, v_j^\epsilon\>_{\cH} = A^{i,j}_{1, \epsilon} +  A^{i,j}_{2, \epsilon}$, where 
	\begin{equ}
		A^{i,j}_{1, \epsilon}  = \< v_i^\epsilon, v_j^\epsilon\>_{\cH} \;,\qquad
		A^{i,j}_{2, \epsilon}  =  \int_{-\infty}^\infty \int_{\R^{d}} \< Dv_i^\epsilon(r,z), v_j^\epsilon\>_\cH \;\xi(dr,dz)\;.
	\end{equ}
	A simple calculation shows that the first term is given by 
	\begin{equs}
		A^{i,j}_{1, \epsilon}  
		&=  {\beta^2} \int_0^{t_i \wedge t_j}
		\int_{\R^{4d}} \epsilon^{-\kappa} R(\tfrac{y_1 -y_2}{\epsilon}) 
		\Bigl( \prod_{m=1}^2 \sigma(u(\tfrac s {\epsilon^2}, \tfrac {y_m} \epsilon)) \Bigr)  \\
		& \hspace{2em} 
		P_{t_i-s} ( x_1 - y_1 ) P_{t_j-s } ( x_2 - y_2 ) g_i(x_1)  g_j(x_2)   \, dx  dy ds\;,
	\end{equs}
	where we used the change of variables
	$s \mapsto \frac s {\epsilon^2}$, $y_i \mapsto \frac{y_i}{\epsilon}$ and the fact that 
	$P_{\frac r {\epsilon^2} } (\tfrac x \epsilon) = \epsilon^d P_r (x)$.
	
	We now write
	\begin{equ}
		\| \<DX_i^\epsilon, v_j^\epsilon\>  - \Sigma_{i,j}\|_2 \leq \big\| \E A^{i,j}_{1, \epsilon} - \Sigma_{i,j} \big\|_2 + \var (A^{i,j}_{1, \epsilon})^{\frac{1}{2}} + \| A^{i,j}_{2, \epsilon} \|_2.
	\end{equ}
	We note that $\E [X^\epsilon_i X^\epsilon_j] = \E A^{i,j}_{1, \epsilon}$ so that $\big| \E A^{i,j}_{1, \epsilon} - \Sigma_{i,j} \big|$
	converges to $0$ by Lemma~\ref{lem:EA_1-covariance_convergence} below.
	
	Regarding the term $A^{i,j}_{2, \epsilon}$, it follows from \eqref{e:defvi} and the identity
	$D_{s,y} \sigma(u(r,z)) = \Sigma(r,z)D_{s,y}u(r,z)$ for a bounded, adapted function-valued process $\Sigma(r,\cdot )$, 
	that, setting
	\begin{equ} \label{eq:defn_A2-tilde}
		\tilde{A}^\tau_{2, \epsilon} (x, s, y) := \int_{s}^{\frac \tau {\epsilon^2}} \int_{\R^{d}} P_{\frac \tau {\epsilon^2} - r}  (\tfrac{x}{\epsilon} - z) \Sigma(r,z) D_{s,y} u(r,z) \xi(dr,dz)\;,
	\end{equ}
	one has the identity
	\begin{equs}
		A^{i,j}_{2, \epsilon}  & =
		\frac {\beta^2} {\epsilon^{\kappa+d}}  \int_{0}^{t_i \wedge t_j} 
		\int_{\R^{4d}} g_i(x_1) g_j(x_2)   \tilde{A}^{t_i}_{2, \epsilon} ( x_1, \tfrac s \epsilon, \tfrac {y_1} \epsilon) \times \\
		& \qquad \qquad \qquad \qquad  \sigma(u_\epsilon(s, y_2)) P_{t_j -s} (x_2 - y_2) R(\tfrac{y_1 -y_2}{\epsilon}) dy dx ds.
	\end{equs}

	The vanishing of $\var (A^{i,j}_{1, \epsilon})$ and $\| A^{i,j}_{2, \epsilon} \|_2$ as $\epsilon \to 0$ will be shown in 
	Proposition~\ref{prop:vanishing_terms_bounds}, concluding the proof of Theorem~\ref{thm:basic-convergence}. 
\end{proof}

It remains to prove Lemma~\ref{lem:EA_1-covariance_convergence} and Proposition~\ref{prop:vanishing_terms_bounds},
which is done in sections~\ref{sec:cov_proof1}  and
\ref{sec:proof_vanishing2} respectively.

\subsection{Convergence of  covariances} \label{sec:cov_proof1}

We drop superscripts, fix $f,g \in \Cc(\R^d)$ and fix $\tau,v >0$. The covariance $ \Sigma_{\tau,v}^{f,g}$ of the limiting terms $Z_\tau(f)$ and $Z_v(g)$ is 
\begin{equs}
	\Sigma_{\tau,v}^{f,g} &:= \beta^2 \nueff^2 J_{\tau, v} (f,g), \\
	J_{\tau, v} (f,g) &:= \int_0^{\tau \wedge v}
	\int_{\R^{4d}} |y_1 -y_2|^{-\kappa}
	P_{\tau-s} ( x_1 - y_1 ) P_{v-s } ( x_2 - y_2 ) f(x_1)  g(x_2) dx dy ds.
\end{equs}

We first observe the following bound on $J_{\tau, v} (f,g)$ in terms of $I_{\Gamma_0}$, defined at \eqref{I-Gamma0}.

\begin{lemma}\label{lem:J_bound_I_Gamma0}
	Let $f, g \in \Cc(\R^d)$, $\tau, v >0$, $d \geq 3$ and  $\lambda \in (0,1)$. Then, 
	\begin{equ}\label{eq:J_bound_I_Gamma0}
	J_{\tau, v} (f,g) \lesssim (\tau \wedge v)^{1-\lambda}\,  I_{\Gamma_0}(f,g; \lambda).
\end{equ}
\end{lemma}
\begin{proof}
	Let $t= \tau \wedge v$. For any fixed $\lambda \in (0,  1)$, with Lemma~\ref{lem:gaussian_time_bounds},
	we can bound the Gaussian kernels  as follows: 	
	\begin{equ} 
		P_{\tau-s} ( x_1 - y_1 ) P_{v-s } ( x_2 - y_2 ) \lesssim (t - s)^{-\lambda}  \,  \prod_{i=1}^2 |x_i -y_i|^{\lambda - d}\;,  
	\end{equ}
	and the claim follows by inserting this into the expression for $J_{\tau, v} (f,g)$ and performing the time integral.\end{proof}

We present one of the lemmas needed for proving the main theorem.
\begin{lemma}\label{lem:EA_1-covariance_convergence}
Assume the assumptions of Theorem~\ref{thm:basic-convergence} on $\kappa$, $R$, $\sigma$ and $\beta < \beta_0$.
	Let $f, g \in \Cc(\R^d)$ and for $\tau, v >0$ let
\begin{equ} \label{eq:defn_A12}
	\begin{aligned}
		A_{1, \epsilon} &
		=   {\beta^2} \int_0^{\tau \wedge v}
		\int_{\R^{4d}} \epsilon^{-\kappa} R(\tfrac{y_1 -y_2}{\epsilon}) 
		\Bigl( \prod_{m=1}^2 \sigma(u(\tfrac s {\epsilon^2}, \tfrac {y_m} \epsilon)) \Bigr) \times \\
		& \hspace{1em} 
		P_{\tau -s} ( x_1 - y_1 ) P_{v-s } ( x_2 - y_2 ) f(x_1)  g(x_2)   \, dx  dy ds.
		\end{aligned}
		\end{equ}
		 Then the following statements hold.
	\begin{itemize}
		\item For any $\lambda \in (0,1)$,  the following bound holds uniform in $\epsilon \in (0,1)$, 
		\begin{equ} \label{eq:EA_1_control}
			\E |A_{1,\epsilon}| \lesssim (\tau\vee v)^{1-\lambda} I_{\Gamma_0}(f,g; \lambda)\;.
		\end{equ}
		\item Furthermore,
		\begin{equ} \label{eq:EA_1_convergence}
			\lim_{\epsilon \to 0} \E \, A_{1,\epsilon} = \beta^2 \nueff^2 \, J_{\tau, v} (f,g) = \Sigma_{\tau,v}^{f,g}\;.
		\end{equ}
	\end{itemize}
\end{lemma}

\begin{proof}
	Note that $I_{\Gamma_0}$ is finite by Example~\ref{lem:diagram0_finite}.
	By the definition of $A_{1,\epsilon}$, one has 
	\begin{equ} \label{eq:EA1epsilon}
		\begin{aligned}
			\E |A_{1, \epsilon}|&
			\le   {\beta^2} \int_0^{\tau \wedge v}
			\int_{\R^{4d}} \epsilon^{-\kappa} \big|R(\tfrac{y_1 -y_2}{\epsilon})\big| 
			\, \E \big[ \sigma(u(\tfrac s {\epsilon^2}, \tfrac {y_1} \epsilon))  \sigma(u(\tfrac s {\epsilon^2}, \tfrac {y_2} \epsilon)) \big] 
			\\
			& \hspace{2em} 
			P_{\tau-s} ( x_1 - y_1 ) P_{v-s } ( x_2 - y_2 ) |f(x_1)  g(x_2)|   \, dx  dy ds. \\
		\end{aligned}
	\end{equ}
	
	By the uniform $L^p$ bounds in \eqref{eq:moment_sigma_u},  $$\sup_{\epsilon, s,y_1, y_2} \E \big[\sigma(u(\tfrac s {\epsilon^2}, \tfrac {y_1} \epsilon))   
	\sigma(u(\tfrac s {\epsilon^2}, \tfrac {y_2} \epsilon))
	\big] < \infty.$$
	Since $\epsilon^{-\kappa} | R(\tfrac x \epsilon)| \lesssim |x|^{-\kappa}$ uniformly in $\epsilon$, it follows that 
	$	 \E |A_{1,\epsilon} | \lesssim \beta^2 J_{\tau, v}(|f|,|g|)$, so that
	\eqref{eq:EA_1_control} follows from Lemma~\ref{lem:J_bound_I_Gamma0}.

	To show \eqref{eq:EA_1_convergence}, we use the stationarity in space of $\{ u(t, x)\}_{x \in \R^d}$ 
	to write 
	$\E A_{1,\epsilon} = \Sigma^\epsilon + B^\epsilon$, where
	\begin{equs}
	\begin{split}
	\Sigma^\epsilon &:=  \beta^2   \int_0^{\tau \wedge v}
	\int_{\R^{4d}} \epsilon^{-\kappa} R(\tfrac{y_1 -y_2}{\epsilon}) 
	\, \E \big[ \sigma(u(\tfrac s {\epsilon^2}, 0))   \big]^2 \\
	& \hspace{8em} 
	P_{\tau-s} ( x_1 - y_1 ) P_{v -s } ( x_2 - y_2 ) f(x_1)  g(x_2)   \, dx  dy ds, 
	\end{split}
	\label{eq:multi_defn_Sigma_epsilon} \\[0.5em]
	\begin{split} 
	B^\epsilon  &:=  \beta^2  \int_0^{\tau \wedge v}
	\int_{\R^{4d}} \epsilon^{-\kappa} R(\tfrac{y_1 -y_2}{\epsilon})  \,
	\cov( \sigma(u(\tfrac s {\epsilon^2}, \tfrac {y_1-y_2} \epsilon)), \sigma(u(\tfrac s {\epsilon^2}, 0)) ) \\
	& \hspace{8em} 
	P_{\tau-s} ( x_1 - y_1 ) P_{v-s } ( x_2 - y_2 ) f(x_1)  g(x_2)   \, dx  dy ds.
	\end{split} \label{eq:multi_defn_B_epsilon}
	\end{equs}

	By \eqref{eq:moment_u}, we have $\E \big[ \sigma(u(\tfrac s {\epsilon^2}, 0) )   \big] \leq C$ uniformly over $s\geq 0$ and $\epsilon >0$. Furthermore by Theorem~\ref{thm:stat_field}, $u(t,x)$ converges in law as $t\to \infty$, so
	$\nueff : = \lim_{t \to \infty} \left| \E [ \sigma (u(t,0))] \right|$ is well defined and for every $s>0$, 
	\begin{equ} \label{eq:conv_eff_var}
		\lim_{\epsilon \to 0}\E \big[ \sigma(u(\tfrac s {\epsilon^2}, 0))   \big]^2
		= \nueff^2\;.
	\end{equ}	
	Similarly, using \eqref{eq:cov_space_estimate}, it follows that for any $ y_1 \neq y_2$, 
	\begin{equ}\label{eq:cov_est_epsilon}
		\cov\Big(\sigma(u(\tfrac s {\epsilon^2}, \tfrac {y_1- y_2} \epsilon)), \sigma(u(\tfrac s {\epsilon^2}, 0)) \Big) \leq C (1 \wedge \tfrac {\epsilon^{\kappa-2}} {|y_1 -y_2 |^{2-\kappa}}) \longrightarrow 0, \quad \text{as}\,\,  \epsilon\to 0\;,
	\end{equ}
	and this quantity is uniformly bounded, again by \eqref{eq:moment_u}.		
	Since $\epsilon^{-\kappa} |R(\tfrac x \epsilon)| \lesssim |x|^{-\kappa}$, it follows from \eqref{eq:conv_eff_var}-\eqref{eq:cov_est_epsilon} that the integrands of both \eqref{eq:multi_defn_Sigma_epsilon}
	and \eqref{eq:multi_defn_B_epsilon} are uniformly bounded by an integrable function.
	It then follows from Lebesgue's dominated convergence theorem that 
	\begin{equ} 
		B^\epsilon \longrightarrow 0, \qquad \Sigma^\epsilon \longrightarrow \beta^2 \nueff^2 J_{\tau, v} = \Sigma^{f,g}_{\tau,v}\;,	\end{equ}
	thus concluding the proof.
\end{proof}

\subsection{Convergence of residual terms}
The focus of this subsection is to provide estimates on
the residual terms $\var (A_{1, \epsilon})$ and $\| A_{2, \epsilon}\|_2$,
where  $A_{1, \epsilon}$ is given in \eqref{eq:defn_A12} and \begin{equ} \label{eq:defn_A12-2}
	\begin{aligned}
			A_{2, \epsilon} & = \epsilon^{-d}
		\beta^2
		\int_{0}^{\tau \wedge v} 
		\int_{\R^{4d}} \epsilon^{-\kappa} R(\tfrac{y_1 -y_2}{\epsilon})  \tilde{A}^{\tau}_{2, \epsilon} ( x_1, \tfrac s \epsilon, \tfrac {y_1} \epsilon) \sigma(u_\epsilon(s, y_2))
		\\
		& \hspace{8em} 
		P_{v -s} (x_2 - y_2)   f(x_1) g(x_2)  dy dx ds,
	\end{aligned}
\end{equ}
with $  \tilde{A}^{\tau}_{2, \epsilon}$  given by ({\ref{eq:defn_A2-tilde}). 

We set
$dx = \prod_{i=1}^4 dx_i$, $dy = \prod_{i=1}^4 dy_i$, $dz = \prod_{i=1}^2 dz_i$,
 For  $f, g :\R^d\to \R$, set
\begin{equs}\label{psi-fg}
\psi_{f,g}(x) &:= f(x_1) g(x_2) g(x_3)  f(x_4),\\ 
\varphi_{f,g}(x)& := f(x_1) g(x_2) f(x_3)  g(x_4),
\label{varphi-fg}\\
Q(x) &:= \int_0^\infty \int_{\R^{2d}} P_r(x - z_1) P_r(z_2) |R(z_1 -z_2)|\, dz\, dr.
\label{Q}
\end{equs}

We first make preliminary estimates.
\begin{lemma}\label{lem:res_terms_estimates_multi}
	For  any test functions $f, g :\R^d\to \R$, one has 
	\begin{equs}[2]
\var(A_{1, \epsilon})& \lesssim \epsilon^{\kappa -2}  t^{1-2\lambda} \sum_
{\substack{i=1,2\\ j= 3,4}}
\int_{\R^{8d}}   \Bigl(
\prod_{k=1}^4 |y_k - x_k|^{\lambda -d}   \Bigr)
|y_i - y_j|^{2-\kappa} \\
&  \hspace{10em}
|y_1- y_2|^{-\kappa} |y_3- y_4|^{-\kappa}  \, 	|\varphi_{f,g}(x)|  \, 
dx dy ds.
	\end{equs}
\end{lemma}

\begin{proof} 
	We follow \cite{gu20_nlmSHE} and begin with
	\begin{equ}
		\begin{aligned}
			\var (A_{1, \epsilon}) &= 
			\beta^2 \int_0^{\tau \wedge v}
			\int_{\R^{8d}} \epsilon^{-\kappa} R(\tfrac{y_1 -y_2}{\epsilon})  \epsilon^{-\kappa} R(\tfrac{y_3 -y_4}{\epsilon}) \,
			\cov\Big( \Lambda_\epsilon (s, y_1, y_2), \Lambda_\epsilon (s, y_3, y_4) \Big)   \\
			& \hspace{1em} 
			\bigg(\prod_{k=1}^2  P_{\tau -s} ( x_k - y_k ) P_{v-s } ( x_{k+2} - y_{k+2}) \bigg) 
			f(x_1) g(x_2) f(x_3)  g(x_4)  \, dx  dy ds
		\end{aligned}
	\end{equ}
	where $
	\Lambda_\epsilon (s, y_1, y_2) := 
	\sigma(u(\tfrac s {\epsilon^2}, \tfrac {y_1} \epsilon))
	\sigma(u(\tfrac s {\epsilon^2}, \tfrac {y_2} \epsilon))$.
	By the Clark--Ocone formula, 
	\begin{equ}
		\begin{aligned}
			\cov\Big( \Lambda_\epsilon (s, y_1, y_2), \Lambda_\epsilon (s, y_3, y_4) \Big) &= 
			\int_0^{\frac s {\epsilon^2}}
			\int_{\R^{2d}} R(z_1 -z_2)  \,\E \Big[ \E [ D_{r,z_1} \Lambda_\epsilon (s, y_1, y_2) | \F_r|]   \\
			& \hspace{4em} \times \E [ D_{r,z_2} \Lambda_\epsilon (s, y_3, y_4) | \F_r|]  \Big]  dz dr.
		\end{aligned}
	\end{equ}
	Using the chain rule to expand $D_{r,z} \Lambda_\epsilon (s, y_1, y_2)$ and the  bounds \eqref{eq:moment_Du}-\eqref{eq:moment_sigma_u}, 
	\begin{equ}\label{eq:est_D_Lambda}
		\| D_{r,z_1} \Lambda_\epsilon (s, y_1, y_2) \|_2 \lesssim P_{\frac s {\epsilon^2} -r} (\tfrac {y_1} \epsilon - z) + P_{\frac s {\epsilon^2} -r} (\tfrac {y_2} \epsilon - z). 
	\end{equ}
	Then, we have
	\begin{equ}\label{eq:bound_cov_Lambda}
		\left| \cov\Big( \Lambda_\epsilon (s, y_1, y_2), \Lambda_\epsilon (s, y_3, y_4) \Big) \right| \lesssim \sum_
		{\substack{i=1,2\\ j= 3,4}} Q(\tfrac{y_i - y_j}{\epsilon}),
	\end{equ}
	thus concluding 	
		\begin{equs}[2]	 [eq:varA1_estimate]
		\var(A_{1, \epsilon})& \lesssim \sum_
		{\substack{i=1,2\\ j= 3,4}}
		\int_{0}^{\tau \wedge v} \int_{\R^{8d}} 
		\bigg(\prod_{k=1}^2  P_{ \tau -s} (y_k - x_k) P_{v -s} (y_{2+k} - x_{2+k})\bigg)  Q(\tfrac{y_i - y_j}{\epsilon}) 
		\\
		&  
		\epsilon^{-\kappa}\big|R(\tfrac {y_1- y_2} \epsilon)\big| \epsilon^{-\kappa}\big|R(\tfrac {y_3- y_4} \epsilon)\big| \, 	
		|\varphi_{f,g}(x)|  \, 
		dx dy ds.
		\end{equs}

Since $\epsilon^{-\kappa} |R(\tfrac{x}{\epsilon})| \lesssim |x|^{-\kappa}$, and by Lemma~\ref{lem:gaussian_cov_est},  we see that
\begin{equ}
	Q(\tfrac{x}{\epsilon}) \lesssim  1 \wedge \epsilon^{\kappa -2}  |x|^{2-\kappa} \lesssim \epsilon^{\kappa -2}  |x|^{2-\kappa}. 
\end{equ}
Plugging these estimates into \eqref{eq:varA1_estimate} and bringing time integral inside, we have 
\begin{align*}
\var(A_{1, \epsilon})& \lesssim \epsilon^{\kappa -2}  \sum_
{\substack{i=1,2\\ j= 3,4}}
\int_{\R^{8d}}   
\int_{0}^{\tau \wedge v} \bigg(\prod_{k=1}^2  P_{\tau-s} (y_k - x_k) P_{v-s} (y_{2+k} - x_{2+k}) \bigg)  ds  
\\
&  \hspace{3em}
|y_i - y_j|^{2-\kappa} |y_1- y_2|^{-\kappa} |y_3- y_4|^{-\kappa}  \, 	|\varphi_{f,g} (x)|
dx\, dy\, ds.
\end{align*}

 Let $t = \tau \wedge v$ and $\lambda \in (0, \frac 1 2)$. We use again the pointwise estimate \eqref{eq:comp_est} from Lemma~\ref{lem:gaussian_time_bounds}. For $r \in \{ \tau, v\}$, we have $r \geq t$, so that $P_{r-s} (x) \lesssim  (t-s)^{-\frac \lambda 2} |x|^{\lambda - d}$ uniformly in $t, s$ and $x \neq 0$. Hence, we estimate as follows
\begin{align*}
\var(A_{1, \epsilon})& \lesssim \epsilon^{\kappa -2}  \sum_
{\substack{i=1,2\\ j= 3,4}}
\int_{\R^{8d}}   \bigg(
\int_{0}^{t} (t-s)^{-2\lambda} \, ds  \bigg)    
\Bigl(\prod_{k=1}^4 |y_k - x_k|^{\lambda -d}  \Bigr)
\\
&  \hspace{2em}
|y_i - y_j|^{2-\kappa} \, 
|y_1- y_2|^{-\kappa} |y_3- y_4|^{-\kappa}  \, 	|\varphi_{f,g}(x)|  \, 
dx dy ds \\
&=\epsilon^{\kappa -2}  t^{1-2\lambda} \sum_
{\substack{i=1,2\\ j= 3,4}}
\int_{\R^{8d}}   \Bigl(
\prod_{k=1}^4 |y_k - x_k|^{\lambda -d}   \Bigr)
|y_i - y_j|^{2-\kappa} \\
&  \hspace{2em}
|y_1- y_2|^{-\kappa} |y_3- y_4|^{-\kappa}  \, 	|\varphi_{f,g}(x)|  \, 
dx dy ds,
\end{align*}
completing the proof.
\end{proof}

\begin{lemma}\label{lem:res_terms_estimates_multi-2}
	For  any test functions $f, g :\R^d\to \R$, one has 
	\begin{equs}
	\| A_{2, \epsilon} \|_2 
	&\lesssim \epsilon^{\frac \kappa 2 -1 } \sqrt t  \bigg[  \int_{\R^{10d}} 
	|\psi_{f,g}(x)| \, |y_1 - y_4|^{-\kappa}|z_1 - y_2|^{-\kappa} |z_2 - y_3|^{-\kappa} \\
	& \qquad 
	\bigg( \int_{0}^{t} P_{{v}-s} ( y_2 - x_2) P_{{v}-s} ( y_3 - x_3)  \int_s^{\tau} 
	P_{{\tau}-r} ( x_1 - y_1  ) P_{{\tau}-r} ( x_4 - y_4  ) 
	\\
	& \qquad
	P_{r-s} ( y_1 - z_1 )  P_{r-s} ( y_4 - z_2 )    
	dr ds \bigg)
	dy dz dx \bigg]^{\frac 1 2}.  
\end{equs}
\end{lemma}

\begin{proof}
 Setting
	\begin{equs}
	B_\epsilon^{\tau} (s, x_1, x_1', y_1, y_1', y_2, y_2') 
		= \E \Big[ \tilde{A}^{\tau}_{2, \epsilon} ( x_1, \tfrac s \epsilon, \tfrac {y_1} \epsilon)
		\tilde{A}^{\tau}_{2, \epsilon} ( x_1', \tfrac s \epsilon, \tfrac {y_1'} \epsilon)
		\sigma(u_\epsilon(s, y_2)) 
		\sigma(u_\epsilon(s, y_2')) 
		\Big].
	\end{equs}
using Minkowski inequality, we obtain:
	\begin{equs}
		\| A_{2,\epsilon}\|_2 & \leq \epsilon^{-d} \int_{0}^{\tau\wedge v} \bigg(
		\int_{\R^{8d}} 
		B_\epsilon^{\tau} (s, x_1, x_1', y_1, y_1', y_2, y_2')
		P_{v-s} (x_2 -y_2) P_{v-s} (x_2' -y_2') \\
		& \hspace{2em}\epsilon^{-\kappa} \big| R(\tfrac{y_1 -y_2}{\epsilon})\big| \epsilon^{-\kappa} \big|R(\tfrac{y_1' -y_2'}{\epsilon})\big| f(x_1) g(x_2) f(x_1') g(x_2') dx dx' dy dy'
		\bigg)^{\frac 1 2} ds.	\end{equs}
	Recalling the definition \eqref{eq:defn_A2-tilde} of $\tilde{A}^\tau_{2, \epsilon}$,	
	first apply It\^{o} isometry, then estimate via H\"{o}lder inequality and the
	moment bounds on $Du$, $\sigma(u)$ from Lemma~\ref{lem:moments}, and finally perform a change of variables,
	\begin{align*}
	& |B_\epsilon^{\tau} (s, x_1, x_1', y_1, y_1', y_2, y_2')  | 
	\\
	& = 
	\bigg|  
	\int_{\frac s {\epsilon^2}}^{\frac \tau {\epsilon^2}} \int_{\R^{2d}} \E \big[
	\Sigma(r,z_1) \Sigma(r,z_2)  D_{\frac s {\epsilon^2}, \frac {y_1} \epsilon} u(r,z_1)
	D_{\frac s {\epsilon^2}, \frac {y_1'} \epsilon} u(r,z_2) \sigma(u(\tfrac s {\epsilon^2}, \tfrac {y_2} \epsilon )) \sigma(u(\tfrac s {\epsilon^2}, \tfrac {y_2'} \epsilon ))
	\big] \\
	& \hspace{4em} P_{\frac \tau {\epsilon^2} - r}  (\tfrac{x_1}{\epsilon} - z_1) 
	P_{\frac \tau {\epsilon^2} - r}  (\tfrac{x_1'}{\epsilon} - z_2) R(z_1 -z_2)\, dz dr
	\bigg|\\
	& \lesssim 
	\int_{\frac s {\epsilon^2}}^{\frac \tau {\epsilon^2}} \int_{\R^{2d}}  
	P_{r-\frac s {\epsilon^2} }  (z_1-\tfrac{y_1}{\epsilon}) P_{r- \frac s {\epsilon^2}}  (z_2-\tfrac{y_1'}{\epsilon} )
	P_{\frac \tau {\epsilon^2} - r}  (\tfrac{x_1}{\epsilon} - z_1) 
	P_{\frac \tau {\epsilon^2} - r}  (\tfrac{x_1'}{\epsilon} - z_2)\big| R(z_1 -z_2)\big|\, dz dr
	\\
	& = \epsilon^{\kappa-2+2d} \int_{ s }^{\tau} \int_{\R^{2d}}  
	P_{r-s}  (z_1-{y_1}) P_{r - s}  ( z_2 - {y_1'} )
	P_{\tau - r}  ({x_1} - z_1) 
	P_{\tau - r}  ({x_1'} - z_2) \epsilon^{-\kappa} \big|R(\tfrac {z_1 -z_2} \epsilon)\big|\, dz dr.
	\end{align*}
	Returning to $\| A_{2,\epsilon}\|_2$ and relabelling integrating variables\footnote{Namely, we relabel $(x_1, x_1', x_2, x_2')$ with $(x_1, x_4, x_2, x_3)$, and $( z_1, z_2, y_1, y_1', y_2, y_2')$ with $( y_1, y_4, z_1, z_2, y_2, y_3)$.} to conclude
		\begin{equs}[2]	
		\| A_{2, \epsilon} \|_2 &\lesssim \epsilon^{\frac \kappa 2 -1 } 	\int_{0}^{\tau \wedge v}  \bigg[ \int_s^{\tau} \int_{\R^{10d}} 
		P_{{\tau}-r} ( x_1 - y_1  ) P_{\tau-r} ( x_4 - y_4  )\\
		& \hspace{1em}  P_{r-s} ( y_1 - z_1 )  P_{r-s} ( y_4 - z_2 )  \,
		\epsilon^{-\kappa}\big|R(\tfrac {y_1- y_4} \epsilon)\big|
		\epsilon^{-\kappa}\big|R(\tfrac {z_1- y_2} \epsilon)\big|
		\epsilon^{-\kappa}\big|R(\tfrac {z_2 - y_3} \epsilon)\big|
		\\
		& \hspace{1em}  P_{{v}-s} ( y_2 - x_2) P_{{v}-s} ( y_3 - x_3) 	\,\big|\psi_{f,g}(x)\big| \,
		dy dz dx dr \bigg]^{\frac 1 2} ds.  \label{eq:A2_estimate}
	\end{equs}

	Since $\epsilon^{-\kappa} |R(\tfrac x \epsilon )| \lesssim |x|^{-\kappa}$, letting $t = v \wedge \tau$,  we have
\begin{equs}
	\| A_{2, \epsilon} \|_2 &\lesssim \epsilon^{\frac \kappa 2 -1 } 	\int_{0}^{t}  \bigg[ \int_s^{\tau} \int_{\R^{10d}}  |\psi_{f,g}(x)| \,
	P_{{\tau}-r} ( x_1 - y_1  ) P_{{\tau}-r} ( x_4 - y_4  ) \\
	& \hspace{6em}  P_{r-s} ( y_1 - z_1 )  P_{r-s} ( y_4 - z_2 )  \,
	|y_1 - y_4|^{-\kappa}|z_1 - y_2|^{-\kappa} |z_2 - y_3|^{-\kappa}  
	\\
	& \hspace{6em}  P_{{v}-s} ( y_2 - x_2) P_{{v}-s} ( y_3 - x_3) 
	dy dz dx dr \bigg]^{\frac 1 2} ds.
\end{equs}
Applying Jensen's inequality (with respect to the normalised measure $\frac{ds}t$) and bringing time integrals inside, 
we have
\begin{equs}
	\| A_{2, \epsilon} \|_2 
	&\lesssim \epsilon^{\frac \kappa 2 -1 } \sqrt t  \bigg[  \int_{\R^{10d}} 
	|\psi_{f,g}(x)| \, |y_1 - y_4|^{-\kappa}|z_1 - y_2|^{-\kappa} |z_2 - y_3|^{-\kappa} \\
	& \qquad 
	\bigg( \int_{0}^{t} P_{{v}-s} ( y_2 - x_2) P_{{v}-s} ( y_3 - x_3)  \int_s^{\tau} 
	P_{{\tau}-r} ( x_1 - y_1  ) P_{{\tau}-r} ( x_4 - y_4  ) 
	\\
	& \qquad
	P_{r-s} ( y_1 - z_1 )  P_{r-s} ( y_4 - z_2 )    
	dr ds \bigg)
	dy dz dx \bigg]^{\frac 1 2},
\end{equs}
obtaining the required inequality.
\end{proof}

\subsection{Elementary estimates}
It remains to work on elementary estimates for the integrals related to the residual terms, for this we employ the Feynman diagrams analysis.
\begin{lemma}\label{lem:multiple_time_int}
	Let $w_1, w_2, x_1, x_2, y_1, y_2 \in \R^d\setminus\{0\}$. For any $\tau, v \geq t>0$ and  $\f12\le \lambda < \frac 2 3$,	
	\begin{equs} 
		 J &= \int_{0}^{t} P_{{v}-s}(w_1) P_{{v}-s}(w_2) \int_{s}^{\tau} P_{{\tau}-r} (x_1) P_{{\tau}-r} (x_2)
		P_{r-s} (y_1)  P_{r-s} (y_2) dr ds 
		\\
		&  \lesssim 
		t^{2-3\lambda} 
		\Bigl(\prod_{i=1}^2 |w_i|^{\lambda - d} |x_i|^{\lambda - d} |y_i|^{\lambda - d}\Bigr)\;.
	\end{equs}
\end{lemma}
\begin{proof}
	Since $\tau, v \geq t$, we can use the pointwise bound  \eqref{eq:comp_est}, yielding
	\begin{equ}\label{eq:lambda_bound_prel_lemma}
		J \lesssim 
		\Bigl(\prod_{i=1}^2 |w_i|^{\lambda - d} |x_i|^{\lambda - d} |y_i|^{\lambda - d}\Bigr)
		\int_{0}^{t} ({v}-s)^{-\lambda} \int_{s}^{\tau} ({\tau}-r)^{-\lambda}  
		(r-s)^{-\lambda}  dr ds.
	\end{equ}
For a constant $c$,
	\begin{equs}
		&\int_{0}^{t} ({v}-s)^{-\lambda} \int_{s}^{\tau} ({\tau}-r)^{-\lambda}  
		(r-s)^{-\lambda}  dr ds
		\\
		&=c
		\int_{0}^{t} ({v}-s)^{-\lambda} ({\tau}-s)^{1-2\lambda}   ds 
		\le	\int_{0}^{t} (t-s)^{1-3\lambda}  
		ds \lesssim t^{2-3\lambda} \;,
	\end{equs}
	this concludes the proof.
	We need $ 1-2 \lambda\le 0$ to replace $\tau -s$ by $t-s$.
\end{proof}

\begin{lemma}\label{lem:simple-graph}
	Let $F\in \Cc(\R^{4d})$. 
	Let $\lambda>0$, $\alpha>-d$, and $ \beta>-d$ satisfy  $\lambda + \alpha <  \min (-\f 12 \beta, \f  12 d, -\beta )$. Then
	the following integral  is finite,	
	\begin{equ}\label{eq:defn_I_Gamma_tilde}
		I_{\tilde \Gamma}(F; \lambda) = \int_{\R^{6d}}
		(|z_1- x_1||z_2-x_3|) ^{\lambda -d}\ ( |z_1 - x_2| |z_2-x_4| )^{\alpha}  \,\, |z_1 - z_2|^{\beta}
		\, 	|F(x)|  \, 
		dx dz.
	\end{equ}
\end{lemma}
\begin{proof} The graph associated with the integral is
	
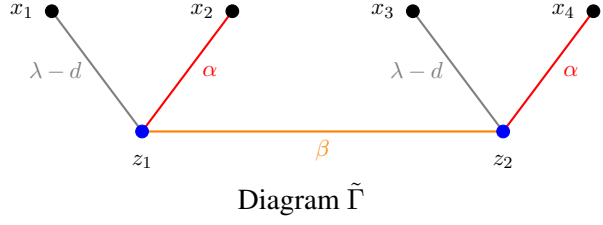
\begin{center}\begin{tikzpicture}[scale=0.8, every node/.style={transform shape}]
		\node (z1) at (1.5,0) {};
		\node (Z1) at (1.5,-0.5) {$z_1$};
		\node (x1) at (0,2) {};
		\node (X1) at (-0.5,2) {$x_1$};
		\node (x2) at (3,2) {};
		\node (X2) at (2.5,2) {$x_2$};
		
		\node (z2) at (7.5,0) {};
		\node (Z2) at (7.5,-0.5) {$z_2$};
		\node (x3) at (6,2) {};
		\node (X3) at (5.5,2) {$x_3$};
		\node (x4) at (9,2) {};
		\node (X4) at (8.5,2) {$x_4$};
		
		\draw[color = gray, thick] (0,2) -- (1.5, 0) node [midway, left] {$\lambda-d \,\,$};
		\draw[color = red, thick] (1.5, 0) -- (3,2) node [midway, right] {$\,\,\alpha$};
		\draw[color = gray, thick] (6,2) -- (7.5,0) node [midway, left] {$\lambda-d\,\,$};
		\draw[color = red, thick] (7.5,0)
		-- (9,2) node [midway, right] {$\,\,\alpha$};
		\draw[color = orange, thick]  (1.5,0) -- (7.5,0) node [midway,below] {$\beta$};
		
		\draw[black, fill=black] (x4) circle (3pt);
		\draw[black, fill=black] (x3)  circle (3pt);
		\draw[black, fill=black] (x2)  circle (3pt);
		\draw[black, fill=black] (x1)  circle (3pt);
		\draw[blue, fill=blue] (z1) circle (3pt);
		\draw[blue, fill=blue] (z2) circle (3pt);
		
		\end{tikzpicture}
		\\
		Diagram $\tilde \Gamma$
	\end{center}
	The small scale degree on each edge of the tree graph is clearly greater than the dimension $-d$, so to apply Lemma~\ref{lem:diagram_integration}  we  only need to check the large scale condition is satisfied. Since 
	$\cV_\ell  =\{\bullet_{x_1}, \bullet_{x_2}, \bullet_{x_3}, \bullet_{x_4}\}$,  
	the tight partitions are:
	$$\begin{aligned}
	\P_1&=\{\cV_\ell, \{\bbullet_{z_1}\}, \{\bbullet_{z_2}\}\},  \qquad \P_2=\{\cV_\ell, \{\bbullet_{z_1},\bbullet_{z_2}\}\},  \\
	\P_3&=\{ \{\cV_\ell, \bbullet_{z_1}\},\{\bbullet_{z_2}\}\},
	\qquad  \P_4=\{ \{\cV_\ell, \bbullet_{z_2}\},\{\bbullet_{z_1}\}\}.
	\end{aligned}
	$$
	A trivial computation shows that 
	$$\begin{aligned}
	\deg_\infty \, \P_1&=2\lambda -2d +2 \alpha +\beta +2d, \qquad \deg_\infty \, \P_2=2\lambda -2d +2 \alpha +d,\\
	\deg_\infty \, \P_3&=\lambda -d + \alpha +\beta +d,
	\end{aligned}$$
	concluding the proof.
\end{proof}

\begin{lemma}\label{lem:diagram1_finite}
	For $F\in \Cc(\R^{4d})$, let
	\begin{equ} \label{eq:defn_I_Gamma1}
		I_{\Gamma_1}(F; \lambda) = \int_{\R^{8d}}
		\Bigl(\prod_{k=1}^4 |y_k - x_k|^{\lambda -d} \Bigr) \, |y_1 - y_4|^{2-\kappa}
		|y_1- y_2|^{-\kappa} |y_3- y_4|^{-\kappa}  \, |F(x)|  \, 
		dx dy,
	\end{equ}	
	\begin{equs} \label{eq:defn_I_Gamma2}
		&I_{\Gamma_2}(F; \lambda)= \int_{\R^{10d}} 
		|y_1 - y_4|^{-\kappa}|z_1 - y_2|^{-\kappa} |z_2 - y_3|^{-\kappa}  \\
		&\hspace{10em}
		|y_1 - z_1|^{\lambda - d}|y_4 - z_2|^{\lambda - d}\, 
		\Bigl(  \prod_{i=1}^4 |x_i - y_i|^{\lambda - d}  \Bigr) |F(x)|dy dz dx.	\quad \quad 
	\end{equs}
	Then $I_{\Gamma_1}(F; \lambda)$ is finite for any $\lambda \in (0,\frac {3\kappa -2} 4)$;  and 
	$I_{\Gamma_2}(F; \lambda)$ is finite for any $\lambda \in (0,\frac {\kappa } 2)$.\end{lemma}
\begin{proof}
	The diagram associated to the first integral is $\Gamma_1$ below, with the labels indicating the
	edge kernel degrees $\deg_0$ (which equals also $\deg_\infty$).
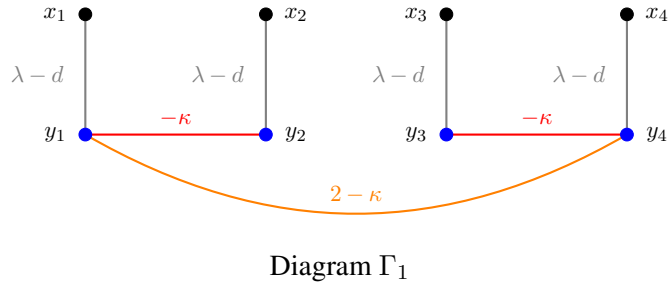
	\begin{center}
		\begin{tikzpicture}[scale=0.8, every node/.style={transform shape}]
		\node at (-0.5,0) {$y_1$};
		\node at (3.5,0) {$y_2$};
		\node at (-0.5,2) {$x_1$};
		\node at (3.5,2) {$x_2$};
		\node at (5.5,0) {$y_3$};
		\node at (9.5,0) {$y_4$};
		\node at (5.5,2) {$x_3$};
		\node at (9.5,2) {$x_4$};
		
		\node at (-0.8, 1) { \textcolor{gray}{$\lambda -d$}};
		\node at (2.2, 1) { \textcolor{gray}{$\lambda -d$}};
		\node at (5.2, 1) { \textcolor{gray}{$\lambda -d$}};
		\node at (8.2 , 1) { \textcolor{gray}{$\lambda -d$}};
		\node at (4.5, -1) {\textcolor{orange}{$2-\kappa$}};
		
		\draw[color = red, thick] (0,0) -- (3,0) node [midway, above] {$-\kappa$};
		\draw[color = gray, thick] (0,0) -- (0,2);
		\draw[color = gray, thick] (3,0) -- (3,2);
		\draw[color = red, thick] (6,0) -- (9,0) node [midway, above] {$-\kappa$};
		\draw[color = gray, thick] (6,0) -- (6,2);
		\draw[color = gray, thick] (9,0) -- (9,2);
		\draw[color = orange, thick]  (0,0) to[out=-30,in=-150] (9,0);
		
		\draw[blue, fill=blue] (0,0) circle (3pt);
		\draw[blue, fill=blue] (3,0) circle (3pt);
		\draw[fill=black] (0,2) circle (3pt);
		\draw[fill=black] (3,2) circle (3pt);
		\draw[blue,fill=blue] (6,0) circle (3pt);
		\draw[blue, fill=blue] (9,0) circle (3pt);
		\draw[fill=black] (6,2) circle (3pt);
		\draw[fill=black] (9,2) circle (3pt);
		
		\end{tikzpicture}\\
		Diagram $\Gamma_1$
	\end{center}
	
	The integral $I_{\Gamma_1}(F; \lambda)$ simplifies after convolution / integrating out $y_2$ and $y_3$. For example,  under our conditions, by \eqref{eq:homog_conv} we have
	$$ \int  |y_2 - x_2|^{\lambda -d} |y_1- y_2|^{-\kappa}	\;dy_2 = C|y_1-x_2|^{\lambda-\kappa}.$$		
		As briefly explained in section~\ref{sec:Feynman}, this procedure collapses diagram $\Gamma_1$ in the sense  replacing 
		$  |y_2 - x_2|^{\lambda -d} |y_1- y_2|^{-\kappa}	$ with $|y_1-x_2|^{\lambda-\kappa}$ corresponds to removing the node $y_2$ in the diagram. Similarly, node $y_3$ is removed.
This corresponds to reduce the integral $I_{\Gamma_1}(F; \lambda)$ as follows:
		\begin{equ}
				I_{\Gamma_1}(F; \lambda) \lesssim  \int_{\R^{6d}}
			(|y_1- x_1||y_4-x_4|) ^{\lambda -d}\ ( |y_1 - x_2| |y_4-x_3| )^{\lambda - \kappa}  \,\, |y_1 - y_4|^{2-\kappa}
			\, 	|F(x)|  \, dy_1 dy_4
			dx. \end{equ}
	Hence, $I_{\Gamma_1}$ is controlled by the integral represented in graph $\tilde \Gamma$ in the proof of Lemma~\ref{lem:simple-graph}, where we take $\alpha = \lambda-\kappa$ and $\beta = 2-\kappa$.  In other words, $I_{\Gamma_1}$ is finite whenever $I_{\tilde \Gamma}$ is, and we conclude the proof by Lemma~\ref{lem:simple-graph}.	

	The diagram associated to $I_{\Gamma_2}$ is as follows:	
	\begin{center}
		\begin{tikzpicture}[scale=0.8, every node/.style={transform shape}]
		\node at (-0.5,0) {$y_1$};
		\node at (4.5,0) {$y_2$};
		\node at (-0.5,2) {$x_1$};
		\node at (4.5,2) {$x_2$};
		\node at (6.5,0) {$y_3$};
		\node at (11.5,0) {$y_4$};
		\node at (6.5,2) {$x_3$};
		\node at (11.5,2) {$x_4$};
		\node at (2.4,-0.8) {$z_1$};
		\node at (8.6,-0.8) {$z_2$};
		
		\node at (-0.8, 1) { \textcolor{gray}{$\lambda -d$}};
		\node at (3.2, 1) { \textcolor{gray}{$\lambda -d$}};
		\node at (6.2, 1) { \textcolor{gray}{$\lambda -d$}};
		\node at (10.2 , 1) { \textcolor{gray}{$\lambda -d$}};
		\node at (5.5, -1.5) {\textcolor{red}{$-\kappa$}};
		
		\draw[color = gray, thick] (0,0) -- (2,-0.5) node [midway, above] {$\lambda -d$};
		\draw[color = red, thick] (2,-0.5) -- (4,0) node [midway, above] {$-\kappa$};
		\draw[color = gray, thick] (0,0) -- (0,2);
		\draw[color = gray, thick] (4,0) -- (4,2);
		\draw[color = gray, thick] (9,-0.5) -- (11,0)  node [midway, above] {$\lambda -d$};
		\draw[color = red, thick] (7,0) -- (9,-0.5)  node [midway, above] {$-\kappa$};
		\draw[color = gray, thick] (7,0) -- (7,2);
		\draw[color = gray, thick] (11,0) -- (11,2);
		\draw[color = red, thick]  (0,0) to[out=-35,in=-145] (11,0);
		
		\draw[blue,fill=blue] (0,0) circle (3pt);
		\draw[blue,fill=blue] (4,0) circle (3pt);
		\draw[fill=black] (0,2) circle (3pt);
		\draw[fill=black] (4,2) circle (3pt);
		\draw[orange,fill=orange] (2,-0.5) circle (3pt);
		\draw[blue,fill=blue] (7,0) circle (3pt);
		\draw[blue,fill=blue] (11,0) circle (3pt);
		\draw[fill=black] (7,2) circle (3pt);
		\draw[fill=black] (11,2) circle (3pt);
		\draw[orange,fill=orange] (9,-0.5) circle (3pt);
		
		\end{tikzpicture}\\
		Diagram $\Gamma_2$
	\end{center}
	We note that, integrating out $y_2\,$, $y_3\,$, $z_1$ and $z_2$, removing the associated nodes in $\Gamma_2$ reduces the graph to $\tilde \Gamma$ where $\alpha = 2\lambda-\kappa$ and $\beta = -\kappa$.
	Applying Lemma~\ref{lem:simple-graph}, we see that also $I_{\Gamma_2}$ is finite and we conclude the proof.	
\end{proof}

With these in place, we are going to recast the  bounds for $ \var( A_{1,\epsilon}) $ and for $ \| A_{2,\epsilon}\|_2$ 
in terms of the form $o(\epsilon) I_{\Gamma_i}$, where
$I_{\Gamma_i}$ were shown to be finite by analysing the associated Feynman diagrams.

\subsection{Main Estimates}\label{sec:proof_vanishing2}
 The main result of the section is:

\begin{proposition}\label{prop:vanishing_terms_bounds}
	Let $f, g \in \Cc(\R^d)$ and $\tau, v >0$. Then, for any given $\lambda$ in the region stated below, the following bounds hold uniformly  in $\epsilon \in (0,1)$:
	\minilab{lemmaBound}\begin{equs}[2]
		\var( A_{1,\epsilon})  &\lesssim \epsilon^{\kappa -2} (\tau \wedge v)^{1-2\lambda} I_{\Gamma_1}(F; \lambda), \qquad
		& \lambda &\in (0,\tfrac 1 2)\;,\qquad\quad
		\label{eq:varA_1_control}\\
		\| A_{2,\epsilon}\|_2  &\lesssim \epsilon^{\frac \kappa 2 -1} (\tau\wedge v)^{\frac 3 2 (1-\lambda)} [ I_{\Gamma_2}(F; \lambda)]^{\frac 1 2}, \quad
		& \lambda &\in  [\tfrac 1 2, \tfrac 2 3);
		\label{eq:L2A_2_control}
	\end{equs}
	where $F=\prod_{i=1}^4 F_i$ where $F_i\in \{f,g\}$.
\end{proposition}
\begin{proof}
Note the integrals  $I_{\Gamma_i}(F; \lambda)$ are finite by Lemmas~\ref{lem:diagram1_finite}.
The proof of the proposition relies crucially on the bounds collected in Lemma~\ref{lem:res_terms_estimates_multi}-\ref{lem:res_terms_estimates_multi-2}. Firstly
\begin{equs}
\var(A_{1, \epsilon})& \lesssim
\epsilon^{\kappa -2}  t^{1-2\lambda} \sum_
{\substack{i=1,2\\ j= 3,4}}
\int_{\R^{8d}}   \Bigl(
\prod_{k=1}^4 |y_k - x_k|^{\lambda -d}   \Bigr)
|y_i - y_j|^{2-\kappa}\times \\
&  \hspace{2em}
|y_1- y_2|^{-\kappa} |y_3- y_4|^{-\kappa}  \, 	|\varphi_{f,g}(x)|  \, 
dx dy ds.
\end{equs}
which by Lemma~\ref{lem:diagram1_finite} is controlled by $I_{\Gamma_1}$:
 $\var(A_{1, \epsilon}) \lesssim \epsilon^{\kappa -2} \, (\tau \wedge v)^{1-2\lambda}I_{\Gamma_1}(F; \lambda)$. 

Fo the second term, we use estimate from Lemma~\ref{lem:res_terms_estimates_multi-2}, to which
we apply Lemma~\ref{lem:multiple_time_int} with some fixed $\lambda \in [\frac 1 2, \frac 2 3)$ to obtain
\begin{equs}
	\| A_{2, \epsilon} \|_2 &\lesssim \epsilon^{\frac \kappa 2 -1 } t^{ \frac 3 2 ( 1- \lambda) } \bigg[  
	\int_{\R^{10d}} |\psi_{f,g}(x)|   
	|y_1 - y_4|^{-\kappa}|z_1 - y_2|^{-\kappa} |z_2 - y_3|^{-\kappa}
	\\
	& \qquad 
	|y_1 - z_1|^{\lambda - d}|y_4 - z_2|^{\lambda - d}\, 
	\Bigl(  \prod_{i=1}^4 |x_i - y_i|^{\lambda - d}  \Bigr) dy dz dx \bigg]^{\frac 1 2}
	\\
	&=  \epsilon^{\frac \kappa 2 -1 } t^{ \frac 3 2 ( 1- \lambda) }  [I_{\Gamma_2} ]^\frac 1 2.
\end{equs}
This concludes the proof of Proposition~\ref{prop:vanishing_terms_bounds}.
\end{proof}

\section{Weak convergence of \TitleEquation{X^\epsilon}{X eps} in negative H\"{o}lder topology}\label{sec:weak_convergence}

Let $\D$ denote the space of smooth functions with compact support, and $\D'$ the dual space
of all distributions. 
Let $\eta\in \D$ denote a test function, not to be confused with the white noise in the beginning of the article, and for $x\in \R^d$ and $\lambda>0$ let $\eta^\lambda_x $ denote its transformation by scaling and translation:
\begin{equ}[e:defScaling]
	\eta^\lambda_x (y) := \lambda^{-d} \eta ( \tfrac{y-x}{\lambda}).
\end{equ}
If $\lambda=1$ or $x=0$, the corresponding index will be dropped.
These operations are $L^1$ invariant. Let $B_a$ denote the open ball of radius $a$ centred at $0$.
Consider the subset of $\C^r$ test functions:
$$  
\B_r := \{ \eta \in \D: \supp(\eta)\subset B_1, \|\eta\|_{\C^r}\le 1 \}.
$$

Let $\alpha<0$, denote $ \lceil -\alpha \rceil \geq -\alpha$ the smallest (positive) integer greater than $-\alpha$.
Let  $\C^\alpha \subset \D^\prime$ denote the subset of locally H\"older continuous distributions. Note that they are not necessarily tempered distributions.
  \begin{definition}\cite[Def.~3.7]{hairer_RegStr}
A distribution  $\zeta $ is said to be in $ \C^\alpha$, if  for every compact set $E \subset \R^d$,
the following \begin{equ}\label{eq:seminorm_Calpha_basic} 
	\| \zeta \|_{\alpha; E} := \sup_{\eta \in \B_r } \sup_{ \lambda \in (0,1)} \sup_{x \in  E}  \,\, \lambda^{-\alpha}| \langle \zeta, \eta^\lambda_x \rangle |,
\end{equ}
where $r = \lceil -\alpha \rceil $, is finite.
\end{definition} These define a family of semi-norms, so $\C^\alpha$ is a Fréchet space
which can be metrised in the usual way, for example by setting 
\begin{equ}\label{dw}
	d_\alpha (\zeta, \zeta'):= \sum_{m \geq 1} 2^{-m}\big(1 \, \wedge \, \| \zeta - \zeta' \|_{\alpha; B_m} \big)\;.
\end{equ}
As metric spaces, $\C^{\alpha'}$ is compactly embedded in $\C^\alpha$ for any $\alpha' > \alpha$.

\subsection{Proof of Theorem~\ref{thm:weak_conv_1}} We recall $\kappa \in (2, d)$ is the decay exponent of the spatial covariance $R$,
$\cU_t(g)$ denotes the solution of (\ref{e:limitEW}) at time $t>0$ tested against a test function $g$. Let $\dot W^\kappa$ denotes the centred Gaussian noise with spatial covariance $|x|^{-\kappa}$,
\begin{equ}  [e:cU_tested]
	\cU_t(g) 
	= \beta \nueff  \int_0^t \int_{ \R^d} \bigg( \int_{ \R^d}
	P_{t-s} (x-y) g(x) dx
	\bigg) dW^\kappa (s,y). 
\end{equ} 
Let us also recall 
\begin{equ}
	X_{t}^{\epsilon, g} = \epsilon^{1-\f \kappa 2}\int_{\R^d}  \bigl( u_\epsilon(t,x)- 1\bigr)g(x) dx.
\end{equ}
	
In this section we show that  $X^\epsilon$ converges weakly to $\cU$, as distribution valued stochastic processes with  H\"{o}lder continuous in time trajectories.
Throughout this section,  $\gamma \in (0, \frac 1 2)$ will parametrise H\"{o}lder continuity in time and $\alpha < 1 - \f \kappa 2 - 2\gamma$ the distributional regularity.

We now prove  Theorem~\ref{thm:weak_conv_1}, which we restate here 
for the convenience of the reader.

\begin{theorem}\label{thm:weak_conv}
Let $d\geq 3$, $\sigma$ be Lipschitz continuous, 
	and suppose that $\xi$ satisfies Assumption~\ref{assump-noise}.	
Then, for any $\gamma \in (0,\frac 1 2)$ and $\alpha < 1 - \frac\kappa2 - 2\gamma$, there exists $\beta_{\gamma,\alpha} >0$ such that, for any $\beta <\beta_{\gamma,\alpha}$ and %
$T>0$, $X^\epsilon$ converges to $\cU$ weakly in $\C^\gamma ([0,T], \C^{\alpha})$.
\end{theorem}

\begin{proof}	
The law of a distribution-valued random variable $\zeta \in \D'$ is uniquely determined by the family of probability distributions of random vectors for the form $(\zeta(g_1), \dots, \zeta(g_n))$,  where $n\in \N$, and $\{g_i\}_{i=1, \dots, n} \subset \D$ are smooth compactly supported functions.

Observe that  $\C^\alpha \subset \D'$, then the law of $\zeta \in  \C^\alpha $ is uniquely determined by such distributions. In Theorem~\ref{thm:basic-convergence}, we have shown that for $\beta < \beta_0(4) \wedge \beta_1(2)$,
for any $0< t_1 \leq t_2 \leq \dots t_n \leq T$, any $\{ g_i \}_{i=1, \dots, n} \subset \D$, as $\epsilon \to 0$,  the following convergence holds in distribution:
\begin{equ} 
	( X^{\epsilon, g_1}_{t_1}, \dots, X^{\epsilon, g_n}_{t_n}) \Rightarrow  ( \cU_{t_1}(g_1), \dots, \cU_{t_n}(g_n)).
\end{equ}
This identifies the law of the accumulation points of $X^\epsilon$, as a random variable in $\C^\gamma( [0,T], \C^\alpha)$.

It remains to show tightness. In the next section we characterise $\C^\alpha$ spaces and then show Proposition~\ref{prop:time-holder_distribution_process-app}, a Kolmogorov-type criterion for $\C^\gamma ([0,T], \C^\alpha(\R^d))$ processes. In
section~\ref{sec:elem_estimates}, we then prove some elementary estimates to be used in Lemmas~\ref{lem:EW_cont_est}-\ref{lem:X_cont_est}. With these,
we identify a range of exponents $\alpha$ and $\gamma$ (cf. Remark~\ref{rem:holder_exponents} below for some heuristics on these) so that $\{X^\epsilon, \epsilon \in (0, 1]\}$ is tight for sufficiently small $\beta < \beta_{\gamma,\alpha}$. Hence, in  Proposition~\ref{prop:holder_Calpha+bound}, items (\ref{item1}) and (\ref{item2a}),  we control in $L^p$
the $\C^\gamma ([0,T], \C^\alpha(\R^d))$ norms of $\cU$ and $X^\epsilon$, uniformly in~$\epsilon$. Finally,  in (\ref{item2b}) of Proposition~\ref{prop:holder_Calpha+bound}, we show $\{X^\epsilon, \epsilon \in (0, 1]\}$ is tight. This concludes the proof of Theorem~\ref{thm:weak_conv_1}.
	\end{proof}
	
\begin{remark}\label{rem:holder_exponents}
	The admissible range of  H\"{o}lder exponents can be derived as follows. One can verify that $\dot W^{\kappa}$ satisfies the
	bound
	\begin{equ}
		\|\langle \dot W^{\kappa},  \lambda^{-1}\eta^{-d}\psi(\lambda^{-1} \cdot , \eta^{-1} \cdot ) \rangle\|_p \lesssim \lambda^{-\f12}\eta^{-\f \kappa 2}\;,\qquad \lambda,\eta \in (0,1]\;,
	\end{equ}
	uniformly over all test functions $\psi$ that are compactly supported in some ball of unit radius and bounded by $1$.
Since this appears on the right-hand side of a heat equation, one would expect the regularity of solutions to be at best improved by 
one time derivative, two space derivatives, or a mixture of both. In other words, one would expect solutions $u$ to be such that, for every
$\theta \in [0,1]$ and test function $\psi$ as before (but with vanishing integral)
	\begin{equ}
		\|\langle u,  \lambda^{-1}\eta^{-d}\psi(\lambda^{-1} \cdot , \eta^{-1} \cdot ) \rangle\|_p \lesssim \lambda^{\theta -\f12}\eta^{2-2\theta-\f \kappa 2}\;,
	\end{equ}
thus suggesting that the range of exponents in Theorem~\ref{thm:weak_conv} is sharp, except that one could probably 
allow for $\gamma \in (-\f12,\f12)$.
\end{remark}

\subsection{Preliminaries} Before proving the technical lemmas and proposition used in the proof of the main theorem, we indicate the topologies of the $\C^\alpha$ space. 
There are a number of useful characterizations for the spaces $\C^\alpha$. It can be characterised in terms of wavelets \cite[Sec~3.1-2]{hairer_RegStr}, but we choose here to use the characterisation from \cite{Cara-Zam} which only
requires us to test against translations and rescalings of one single essentially arbitrary test function.

For any function $\varphi:\R^d\to \R$, we introduce the notation for its rescaled version at scale $2^{-n}$:
\begin{equ}
\varphi^{(n)} = 2^{nd} \varphi (2^n \cdot)\;,\qquad
\varphi_x^{(n)} = 2^{nd} \varphi \bigl(2^n (\cdot-x)\bigr)\;.
\end{equ}
We also write $\bar \D \subset \D$ for the set of test functions $\varphi$ such that $\supp \varphi \subset B_1$
and $\int \varphi = 1$.
We now fix such a $\varphi \in \bar \D$ once and for all for the remainder of the article. Our main tool
is the following Kolmogorov-type criterion.

\begin{theorem}[{\cite[Theorem 12.4]{Cara-Zam}}]\label{thm:critHolder}
Let $\zeta \in \D'$ be a distribution and let $\alpha \in (-\infty, 0]$.
Then, one has $\zeta \in \C^\alpha$ if and only if, for any compact set $E$ and uniformly over $k\in \N$,
    \[
   \sup_{x\in E} |\langle\zeta,\varphi^{(k)}_x\rangle| \lesssim2^{-k \alpha}.
    \]
Furthermore, the semi-norms of $\zeta$ can be estimated by
\begin{equ}[e:boundSeminorm]
\|\zeta\|_{\alpha;E} \leq C \sup_{x \in E_2}\sup_{ k\in \N} 
\; {2^{k\alpha}|\<\zeta,\varphi^{(k)}_x\>| }
\end{equ}
where $C = C(\varphi, \alpha, d)$ is an explicit constant and $E_2=\{z: d(z,E)\le 2\}$.
\end{theorem}

Note that the test function $\varphi$ is \textit{fixed}, yet control of $\langle\zeta,\varphi^{(k)}_x\rangle$
yields a similar control for \textit{all} test functions as in \eqref{eq:seminorm_Calpha_basic}.

\begin{lemma}\label{lem:random_C^alpha}
Let $ p\geq1$ and $\alpha < 0$. Let $\zeta$ be a random distribution.	
	If for any compact $E$ there exists a constant $C_E$ such that
	\begin{equ} \label{eq:moment_cond_Calpha_crit}
		\sup_n 2^{\alpha n}\norm{\sup_{x \in E} |\langle \zeta , \varphi^{(n)}_x \rangle|}_p \leq C_E,
	\end{equ}
	then, for any $\alpha' < \alpha$,  $\zeta \in \C^{\alpha'}$ almost surely and $\Big\| \|\zeta \|_{\alpha', E} \Big\|_p \leq \sum_{n =1}^\infty  2^{(\alpha' - \alpha)n} C_{E_2}$. 
\end{lemma}

\begin{proof}
Setting $\zeta_n(x) = \langle \zeta , \varphi^{(n)}_x \rangle$, it follows from \eqref{e:boundSeminorm} 
and \eqref{eq:moment_cond_Calpha_crit} that 	
	\begin{equ}	\norm{\|\zeta \|_{\alpha', E}}_p
		= \norm{\sup_{n \geq 1} \, \sup_{x \in E_2} 2^{\alpha'n} |\zeta_n(x)|  }_p 
		\leq \sum_{n \geq 1} 2^{(\alpha' - \alpha)n}  \norm{\sup_{x \in E_2} 2^{\alpha n} |\zeta_n(x)|}_p,  
	\end{equ}
	as stated in the lemma.
\end{proof}
	
Let us now define a space of distribution-valued processes with spatial H\"{o}lder regularity at lattice points controlled  when tested against translates of the functions $\varphi^{(n)}$. 

\begin{definition}\label{dis-class}
	A random distribution $\zeta$ is said to belong to $\C^{\alpha_0}_p$, with respect to the $L^{p}(\Omega)$ norm, if there exists $C>0$ such that the following estimates hold  for any  $n\geq 1$:
	\begin{equs}
		\label{eq:lattice+time_control}
		\sup_{x \in  \R^d} 
		\bigl\|\langle \zeta, \varphi^{(n)}_x \rangle\bigr\|_{p} &\leq C 2^{-n\alpha_0} , \\
		\label{eq:spatial_control}
		\sup_{\substack{x, y \in  \R^d,\\
				0<|x-y|\leq 2^{-n}}}  
		\bigl\|\langle \zeta, \varphi^{(n)}_x - \varphi^{(n)}_y \rangle \bigr\|_{p} &\leq C 2^{-n(\alpha_0-1)} |x-y|,
	\end{equs}
	where $|\cdot|$ denotes the sup norm on $\R^d$. The smallest such constant $C$ is denoted by $\|\zeta\|_{\C^{\alpha_0}_p}$. 
	We furthermore write $\C^{\gamma_0, \alpha_0}_p$ as a shorthand for the space $\C^{\gamma_0}(\R_+, \C^{\alpha_0}_p)$.
\end{definition}

\begin{remark}
Note that if \eqref{eq:lattice+time_control} holds uniformly over the class of test functions 
$\bar \D$, then \eqref{eq:spatial_control} follows immediately just like in the proof
of Lemma~\ref{lem:EW_cont_est} below.
\end{remark}

The definition is tailored to the embeddings in Proposition \ref{prop:time-holder_distribution_process-app}  below. This allows }for redistribution of regularity across time and space, leading to the tightness criterion 
(Proposition \ref{prop:tightness-general}) for distribution valued processes with positive H\"{o}lder continuous trajectory in time. This is very similar to other Kolmogorov-type criteria
as for example in \cite[Thm~10.7]{hairer_RegStr}, \cite[Prop.~2.4]{cosco2019spacetime_fluct},
or \cite[Thm 1.1]{furlan2017tightness}, but 
given the range of relation between the spatial and temporal regularity exponents, it doesn't seem to be an immediate
corollary of one of these results.

\begin{proposition}\label{prop:time-holder_distribution_process-app}
	Let $\alpha_0<0$, $p > d$, and $\gamma_0 > \frac 1 {p}$. Then,
	for any $\gamma \in (0,\gamma_0  - \frac 1 p)$ and any $\alpha < \alpha_0  - \frac d p$,
one has the continuous embeddings
	\begin{equ}
	 \C^{\alpha_0}_p \subset L^p(\Omega,\C^{\alpha})\;,\qquad
	 \C^{\gamma_0,\alpha_0}_p \subset L^p(\Omega,\C^\gamma(\R_+,\C^{\alpha}))\;.
	\end{equ}
	Here, we view $\C^\alpha$ as the metric space equipped with the distance $d_\alpha$ as in \eqref{dw}.
\end{proposition}

\begin{remark}
	If we consider a process on a finite interval and set $\C^{\gamma_0, \alpha_0}_{p, T} := \C^\gamma([0, T], \C^{\alpha_0}_p)$, the same proof clearly holds and $\C^{\gamma_0, \alpha_0}_{p, T}$ is continuously embedded in  $L^p(\Omega,\C^\gamma([0, T],\C^{\alpha}))$. 
\end{remark}

\begin{proof} 
We will show that 
\begin{equ}[e:wantedBoundzeta]
\| d_\alpha(\zeta,0)\|_p \lesssim \|\zeta\|_{\C^{\alpha_0}_p}\;,
\end{equ}
which not only implies the first embedding, but also implies that 
\begin{equ}
\| d_\alpha(\zeta_r,\zeta_t)\|_p
= \| d_\alpha(\zeta_r-\zeta_t,0)\|_p
\lesssim |t-r|^{\gamma_0} \|\zeta\|_{\C^{\gamma_0,\alpha_0}_p}\;,
\end{equ}
so that the second embedding then follows from the usual Kolmogorov continuity theorem.

Write now $C_x^{(n)}$ for the cube of sidelength $2^{1-n}$ centred at $x$ and 
set $\zeta^{n}(x) = \langle \zeta,\varphi_x^{(n)}\rangle$.
We note that the rescaled process 
$\tilde \zeta^n_{x}(z) := \zeta^{n}(2^{-n}(z-x))$ satisfies
$\|\tilde \zeta^{n}_{x}(y) - \tilde \zeta^{n}_{x}(z)\|_{p} \leq C 2^{-n \alpha_0} |y-z|$, 
so that Kolmogorov's criterion yields 
\begin{equ}
\Bigl\|\sup_{y \in C_0^{(0)}} |\tilde \zeta^{n}_{x}(y)- \tilde \zeta^{n}_{x}(0)|\Bigr\|_p
\lesssim  2^{-n \alpha_0} \;.
\end{equ}
When combined with \eqref{eq:lattice+time_control}, this in turn yields 
$\bigl\|\sup_{y \in C_x^{(n)}} |\zeta^{n}(y)|\bigr\|_p \lesssim  2^{-n \alpha_0}$.
Given a compact set $E \subset \R^d$, write  $E_n \subset 2^{-n}\mathbb{Z}^d$ for the smallest set so that $E \subset \bigcup_{x \in E_n} C_x^{(n)}$.
It follows that 
\begin{equ}
\Bigl\|\sup_{y \in E} |\zeta^{n}(y)|\Bigr\|_p
\le  (\#E_n)^{\frac 1 {p}}  \sup_{x \in  E_n}
\Bigl\|\sup_{y \in C_x^{(n)}} |\zeta^{n}(y)|\Bigr\|_p.
\end{equ}
To summarise, we have the estimate:
\begin{equ}\label{eq:moment_sup_compact-app}
\Bigl\|\sup_{y \in E} | \langle \zeta, \varphi^{(n)}_y \rangle|\Bigr\|_p
 \lesssim  \|\zeta\|_{\C^{\gamma_0, \alpha_0}_p}
 (1+2^{n}\diam(E))^{\frac dp}2^{-n \alpha_0}\;.
\end{equ}
In particular, we can use Lemma~\ref{lem:random_C^alpha} to obtain
\begin{equ}[e:implication]
\bigl\|\|\zeta\|_{\alpha;E}\bigr\|_p
\lesssim \sum_n 2^{(\alpha - \alpha_0)n} \bigl\|\sup_{y \in E_2} 2^{n \alpha} |\zeta^{n}(y)|\bigr\|_p
\lesssim (1+\diam(E_2))^{\frac dp}\;.
\end{equ}
In order to obtain \eqref{e:wantedBoundzeta}, it now suffices to use \eqref{e:implication} to bound
	\begin{equ}
	\norm{d_\alpha (\zeta, 0)}_p \leq 
	\sum_{m \geq 1} 2^{-m} \big(1 \, \wedge \, \norm{ \|  \zeta \|_{\alpha; B_{m+2} }}_p \big) 
	\lesssim \sum_{m \geq 1} 2^{-m}(1 + \diam(B_{m+2})^{\frac 1 {p}}  ) \lesssim 1\;,
	\end{equ}
as required.
\end{proof}

As a corollary, we obtain the following tightness criterion. 

\begin{proposition}[Tightness criterion]\label{prop:tightness-general}
	Let $\{\zeta^k, k \in \mathcal{I}\}$ be a family of distribution-valued stochastic processes such that $\sup_k\|\zeta^k\|_{\C^{\gamma_0, \alpha_0}_p} <\infty$ for some  $\alpha_0<0$, $p > d$,  $\gamma_0 > \frac 1 {p}$. 
	 Then, for any $T >0$, any $\alpha < \alpha_0 - \frac d p$, and any $ \gamma < \gamma_0 - \frac 1 p$, $\{\zeta^k, k \in \mathcal{I}\}$ is tight in $\C^\gamma([0, T], \C^\alpha)$.
\end{proposition}

\begin{proof} 
For  $ \gamma < \gamma_0 - \frac 1 p$, by  Proposition~\ref{prop:time-holder_distribution_process-app}, $ \norm{ \| \zeta \|_{\C^\gamma ([0,T], \C^\alpha)}  }_p$ is uniformly bounded in $k$. Tightness follows from the fact that
$\C^{\gamma'}([0,T], \C^{\alpha'})$ is compactly embedded in $\C^{\gamma}([0,T], \C^{\alpha})$ for any $0< \gamma' < \gamma$ and $\alpha' < \alpha$.
\end{proof}

\subsection{Some elementary estimates}\label{sec:elem_estimates}
We collect the following bounds.

\begin{lemma}\label{eq:J_kappa_delta_bound}
For any compactly supported bounded $f: \R^d\to \R^m$, define \begin{equ} \label{J}
		J_{\kappa, \delta} (f) := \int_{\R^{4d}} \bigg( \prod_{i=1}^2 | f(x_i)| |x_i - y_i|^{\delta-d} \bigg) |y_1 - y_2|^{-\kappa} dy dx .   
\end{equ}
Denoting $M$  the diameter of the support $f$, then for $\delta \in (0,\f \kappa 2)$, $J_{\kappa, \delta} (f)  \lesssim \, M^{2d+2\delta-\kappa} \| f\|_\infty ^2$.
Furthermore, one has the scaling relation
$	J_{\kappa, \delta}( g_x^\lambda) = \lambda^{2\delta - \kappa}J_{\kappa,  \delta}(g)$,
with $g^\lambda_x (y)=  \lambda^{-d} g(\frac {y-x} \lambda)$ as in \eqref{e:defScaling}.

\end{lemma}
\begin{proof}
For a universal constant $c$, \begin{equ}
	J_{\kappa, \delta} (f) = c\int_{\R^{2d}} |f(x_1)| | f(x_2)| |x_1-x_2|^{2\delta-\kappa}\,dx, \end{equ}
 we have used the fact that
 $\int_{\R^d} |x - y |^\alpha |y-w|^\beta dy = c_{\alpha, \beta} |w-x|^{\alpha + \beta +d}$, provided $\alpha+\beta+d<0$ and $\alpha<-d$, $\beta<-d$.  From this, the scaling relation follows. The inequality then follows trivially.\end{proof}
\begin{definition}
For $g: \R^d\to \R$, set $g^\lambda (y) := \lambda^{-d} g ( \tfrac{y}{\lambda})$. Define
\begin{equs}\label{eq:defnIA_IB}
		I_A(r,t) &= \int_{r}^{t} \int_{\R^{2d}} \bigg( \prod_{i=1}^2 \int_{\R^{d}} P_{t-s} (x_i - y_i) 
		|g^\lambda(x_i)| dx_i \bigg) |y_1 - y_2|^{-\kappa} dy ds, \\
		I_B(r,t) &= \int_0^{r} \!\!\int_{\R^{2d}} \Bigl( \prod_{i=1}^2 \int_{\R^{d}} \big| P_{t-s} (x_i - y_i) - P_{r-s} (x_i - y_i)  \big|
		|g^\lambda(x_i)| dx_i \Bigr) |y_1 - y_2|^{-\kappa} dy ds.
\end{equs}
\end{definition}

 \begin{lemma}\label{lem:bounds_IAB}
 	For any $\delta \in (0, 1)$, $\gamma \in (0, 1-\delta)$, $\lambda \in (0,1]$, $0< r \leq t$,  and $g \in \C^{1}_c$, we have:
	\begin{equs}
		I_A(r,t) & \lesssim  M^{2d+2\delta-\kappa}\; \, |t -r|^{1  - \delta} \lambda^{2 \delta - \kappa  }  \, \| g\|_\infty ^2,  \label{eq:est_IA}\\
		I_B(r,t) &\lesssim  M^{2d+2\delta-\kappa} \,  r^{1-\delta - \gamma} |t -r|^{\gamma} \lambda^{2 \delta - \kappa  }  \, \| g\|_\infty ^2, \label{eq:est_IB}
	\end{equs}
	where  $M$ denotes the diameter of the support of $g$.
\end{lemma}
\begin{proof}
By change of time variable and the heat kernel estimate \eqref{eq:comp_est} with $\lambda $ there taken to be $\delta$, 
	\begin{equs}
		I_A(r,t) &=  \int_{\R^{4d}}  \int_{0}^{t-r} \Bigl(\prod_{i=1}^2  P_{s} (x_i - y_i)\Bigr) ds \, 
		\Bigl(\prod_{i=1}^2  |g^\lambda(x_i)|\Bigr) |y_1 - y_2|^{-\kappa} dy dx \\
		&\lesssim \int_{0}^{t-r} s^{-\delta} ds \, 
		\int_{\R^{4d}}    \Bigl(\prod_{i=1}^2    |g^\lambda(x_i)| |x_i-y_i|^{\delta -d }   \Bigr) |y_1 - y_2|^{-\kappa} dy dx \\
		&\lesssim  (t-r)^{1-\delta} J_{\kappa, \delta}(g^\lambda)
		=  (t-r)^{1-\delta}  \lambda^{2\delta - \kappa}J_{\kappa,  \delta}(g).
	\end{equs}
	Since $\delta<1<\f \kappa 2$,  \eqref{eq:est_IA} holds by Lemma \ref{eq:J_kappa_delta_bound}. 
	
	To estimate $I_B$,  we use H\"{o}lder continuity of the heat kernel,  in time. For any $t \geq r>0$ 
	we obtain from the second part of Lemma~\ref{lem:gaussian_time_bounds}.
	\begin{equ}
			|P_t(x) - P_r(x)|
			\lesssim (t-r)^{\f \gamma 2 } r^{-\f {\gamma+\delta}  2}|x|^{\delta-d}\;.
	\end{equ}
Inserting this bound into \eqref{eq:defnIA_IB}, we obtain
	\begin{equs}
		I_B(r,t) & \lesssim (t-r)^\gamma   \int_0^r (r-s)^{- \gamma-\delta} \int_{\R^{4d}}  \Bigl(\prod_{i=1}^2  |x_i - y_i|^{\delta-d}  |g^\lambda(x_i)| \Bigr) |y_1 - y_2|^{-\kappa} dy dx ds\\
		& \lesssim  (t-r)^\gamma r^{1-\gamma-\delta}  
	 	J_{\kappa, \delta}(g^\lambda)\;,
	\end{equs}
so that \eqref{eq:est_IB} follows.	
\end{proof}

\subsection{Tightness in \TitleEquation{\C^\gamma( [0,T], \C^\alpha)}{Cgamma([0,T],Calpha)}}

In the next two lemmas we gather the time increment and spatial regularity bounds for $\cU$ and $X^\epsilon$,  allowing us to identify $\gamma$ and $\alpha$ such that the required tightness holds in $\C^\gamma( [0,T], \C^\alpha)$.
Recall the notation $g^\lambda_x (y) := \lambda^{-d} g ( \f{y-x}{\lambda}).$

\begin{lemma}\label{lem:EW_cont_est}
	Let $p\geq 2,\kappa \in (2, d)$,  $\gamma \in (0, \f 1 2)$,  $\alpha+2\gamma < 1-\f \kappa 2$.
 Set  $\cU_{r,t} =  \cU_t - \cU_r$.
	Then,  the following holds, for any $0< r <t \leq T$, $x, y \in \R^d$ and $\lambda >0$, and any  $g \in \C^{1}_c$,
	\begin{equs}
		\label{eq:EW_cont_est_p}
		\| \cU_{r,t}( g^\lambda_x )    \|_p &\lesssim  \beta \nueff  M^{d+\alpha} |t -r|^{\gamma} \lambda^{\alpha}   \|g\|_\infty\;, \\
		\label{eq:EW_cont_est_p_spatial}
		\| \cU_{r,t}( g^\lambda_x - g^\lambda_y )    \|_p &\lesssim \beta \nueff  M^{d+\alpha}  |t -r|^{\gamma} \lambda^{\alpha -1}   |g|_{\Lip} |x-y|\;,
	\end{equs}
	where  $M$ is the diameter of the support of $g$. \end{lemma}

\begin{proof}
	By Gaussianity, it is sufficient to work with $p=2$.
By the invariance in law of $\cU$ under spatial translation: $ \cU_t(g^\lambda_x) \eqlaw \cU_t( g^\lambda)$, 
so we may take $x=0$ for the first inequality. 
Dividing the integration region $[0,  t] =  [ 0, r ]\cup [ r , t ]$, using It\^{o} isometry, we compute the second moment of $\cU_{r,t}(g^\lambda)$. Firstly, recall \eqref{e:cU_tested},
\begin{equ}  
	\cU_t(g) 
	= \beta \nueff  \int_0^t \int_{ \R^d} \bigg( \int_{ \R^d}
	P_{t-s} (x-y) g(x) dx
	\bigg) dW^\kappa (s,y). 
\end{equ} 
Then,	\begin{equs}\label{U}
		&\| \cU_t( g^\lambda)  -   \cU_r( g^\lambda)    \|_2^2 \\
		& = \beta^2 \nueff^2 \E \bigg( \int_{\R^{1+d}} 	\int_{\R^d}\Bigl( \1_{[0,t]}(s) P_{t-s} (z - y) - \1_{[0,r]}(s) P_{r-s} (z - y)\Bigr) g^\lambda(z) dz  \;  dW^\kappa(s,y) \bigg)^2 \\
		&  \leq\beta^2 \nueff^2 \big( I_A(r,t) + I_B(r,t) \big), 
	\end{equs}
	where $I_A$ and $I_B$ are defined in \eqref{eq:defnIA_IB} above. According to Lemma~\ref{lem:bounds_IAB}, for $\tilde\gamma\in (0,1- \delta$), $\delta, \lambda \in (0,1)$, we have:
	\begin{equs}
	\| \cU_{r,t}( g^\lambda)    \|_2 
	&\leq \beta \nueff( |t -r|^{1  - \delta} +r^{1- \delta -\tilde \gamma} |t -r|^{\tilde\gamma})^{\f 12}
	M^{d+\delta-\kappa/2} \lambda^{ \delta -\f \kappa 2  }  \, \| g\|_\infty\\
	& \lesssim  \beta \nueff  \;  M^{d+\delta-\kappa/2} \;|t -r|^{\f{\tilde\gamma} 2} \lambda^{  \delta -\f \kappa 2  }  \, \| g\|_\infty.
	\end{equs}
Inequality (\ref{eq:EW_cont_est_p}) follows from setting $\gamma=\tilde \gamma/2$ and $\alpha:=\delta-\f \kappa 2$.
The constraint becomes $\alpha+2\gamma<1 -\f \kappa 2$.
	
In order to obtain the bound on $ \| \cU_{r,t}( g^\lambda_x - g^\lambda_y )    \|_p$, we first assume that $|x-y| \ge \lambda M$. It then trivially follows from applying \eqref{eq:EW_cont_est_p} separately to $g_x^\lambda$ and $g_y^\lambda$, whose diameter of support is bounded by $M$, and observing that $\lambda^{\alpha}\le \lambda^{\alpha-1} |x-y| M^{-1}$, as well as $\|g\|_\infty \lesssim M |g|_\Lip$.

In the case $|x-y| \le \lambda M$, it then suffices to note that
	\begin{equ}
	g_x^\lambda - g_y^\lambda = \tilde g_x^\lambda\;,\qquad \tilde g = g - g_{\frac{y-x}\lambda}\;,
	\end{equ}
	and that $\|\tilde g\|_\infty \le \lambda^{-1}|g|_\Lip |x-y|$. Furthermore, the diameter of the support of $\tilde g$ is bounded by twice that of $g$,
	so the bound follows again from \eqref{eq:EW_cont_est_p}.
	\end{proof}

\begin{lemma}\label{lem:X_cont_est} 
Let $d\geq 3$, $p \geq 2$, $\sigma$ be Lipschitz continuous, $\beta < \beta_0(p)$, suppose that $\xi$ satisfies Assumption~\ref{assump-noise}, and set  $X^\epsilon_{r,t} =  X^\epsilon_t - X^\epsilon_r$.
	For any $\gamma \in (0, \f 1 2)$, any $\alpha < 1 - \f \kappa 2 - 2\gamma$ and any $g \in \C^{1}_c$
	with diameter of its support equal to $M$, 
	the following holds for any $0\leq r <t \leq T$, $x,y \in \R^d$ and $\lambda >0$, 
	\begin{equs}
		\label{eq:Xeps_cont_est_p}
		\sup_{\epsilon \in (0,1)} \| X_{r,t}^{\epsilon,  g^\lambda_x}  \|_p &\lesssim 
	\beta   M^{d+\alpha}\, |t -r|^{\gamma} \lambda^{\alpha}  \, \| g\|_\infty\;,  \\
		\label{eq:Xeps_cont_est_p_spatial}
		\sup_{\epsilon \in (0,1)} \| X_{r,t}^{\epsilon,  g^\lambda_x-g^\lambda_y}  \|_p  &\lesssim  \beta M^{d+\alpha}|t -r|^{\gamma} \lambda^{\alpha-1}   |g|_{\Lip} |x-y|\;,
	\end{equs}
	uniformly over $\eps \in (0,1]$.
	\end{lemma}
\begin{proof}	Given spatial stationarity of $u_\epsilon(t,x)$, we only need to work with $ g^\lambda$, without the shifts. 	
	Note that
	\begin{align*}
		&X_t^{\epsilon,  g^\lambda } - X_r^{\epsilon,  g^\lambda } = 
		\epsilon^{1 - \frac \kappa 2}
		\int_{\R^{d}} \big( u_\epsilon(t,x) - u_\epsilon(r,x)\big)  g^\lambda (x) dx = 
		\beta \epsilon^{1 - \frac \kappa 2} \big( \A^{\epsilon}_{r, t} + \B^{\epsilon}_{r, t} \big),  \; \hbox{where}\\
		&\A^{\epsilon}_{r, t} =  \int_{\frac{r}{\epsilon^2}}^{\frac{t}{\epsilon^2}}
		\int_{\R^{d}} \sigma(u(s,y)) \int_{\R^{d}} P_{\frac{t}{\epsilon^2} -s}  (\tfrac{ x}{\epsilon} -y)   g^\lambda (x) dx \, \xi(ds,dy), \\
		&\B^{\epsilon}_{r, t} =
		\int_0^{\frac{r}{\epsilon^2}} \int_{\R^{d}}  \sigma(u(s,y)) 
		\int_{\R^{d}} \big( P_{\frac{t}{\epsilon^2} -s}  (\tfrac{ x}{\epsilon} -y) - P_{\frac{r}{\epsilon^2} -s}  (\tfrac{ x}{\epsilon} -y) \big)   g^\lambda (x) dx \, \xi(ds,dy). 
	\end{align*}
	In contrast to $\cU_t(g)$, these processes are not Gaussian, but we can estimate their $p$th moments
	via the BDG and Minkowski inequalities, followed by a change of variables:
	\begin{align*}
		&\E [ (\A^{\epsilon}_{r, t})^p] \lesssim \E [\langle \A^\epsilon_{ r, \cdot } 
		\rangle_t^{\frac p 2} ]  \\
		& \hspace{1em}=
		\E \bigg[ \Bigl(   
		\int_{\frac{r}{\epsilon^2}}^{\frac{t}{\epsilon^2}} \int_{\R^{4d}} 
		\Bigl( \prod_{i=1}^2     
		\sigma(u(s,y_i))  P_{\frac{t}{\epsilon^2} -s}  (\tfrac{ x_i}{\epsilon} -y_i)   g^\lambda (x_i)  
		\Bigr) R(y_1 -y_2) dx dy ds
		\Bigr)^{\frac p 2} \bigg] \\
		& \hspace{1em}\leq \bigg( \int_{\frac{r}{\epsilon^2}}^{\frac{t}{\epsilon^2}} \int_{\R^{4d}} 
		\Bigl\| \prod_{i=1}^2 \sigma(u(s,y_i))  \Bigr\|_{\frac p 2} 
		\Bigl(\prod_{i=1}^2  P_{\frac{t}{\epsilon^2} -s}  (\tfrac{ x_i}{\epsilon} -y_i)  |g^\lambda (x_i)| \Bigr)  |R(y_1 - y_2)|  dx dy ds
		\bigg)^{\frac p 2}  \\
		& \hspace{1em}\leq 
		\bigg( \epsilon^{\kappa -2} \int_{r}^{t} \int_{\R^{4d}} 
		\Bigl\| \prod_{i=1}^2 \sigma(u_\epsilon(s,y_i))  \Bigr\|_{\frac p 2} 
		\Bigl(\prod_{i=1}^2  P_{t -s}  (x_i -y_i)  |g^\lambda (x_i)| \Bigr)  \epsilon^{-\kappa} \big| R(\tfrac{y_1 - y_2}{\epsilon}) \big| dx dy ds
		\bigg)^{\frac p 2}.  
	\end{align*}
	Recalling $\epsilon^{-\kappa} \bigl|R(\tfrac{y_1 - y_2}{\epsilon})\bigr| \lesssim |y_1 - y_2|^{-\kappa}$ and the moment bounds \eqref{eq:moment_sigma_u}, we have
	\begin{equ} 
		\|\A^{\epsilon}_{r, t} \|_p \lesssim  \epsilon^{\frac\kappa 2 -1 } 
		\bigg( 
		\int_{r}^{t} \int_{\R^{4d}}
		\Bigl(\prod_{i=1}^2  P_{t -s}  (x_i -y_i)  |g^\lambda (x_i)| \Bigr)  |y_1 - y_2|^{-\kappa}  dx dy ds\bigg)^{\frac 1 2} = \epsilon^{\frac\kappa 2 -1 }  I_A(r,t)^{\frac 1 2}.
	\end{equ} 
	Following the same steps, we also have 
	$
	\|\B^{\epsilon}_{r, t} \|_p \leq \| \langle \B^{\epsilon}_{\cdot, t}   \rangle_r^{\frac 1 2} \|_p \lesssim \epsilon^{\frac\kappa 2 -1 }  I_B(r,t)^{\frac 1 2}$.
	Combining the two bounds, we have
	\begin{equ}
		\| X_{t}^{\epsilon,  g^\lambda} 
		- X_r^{\epsilon,  g^\lambda}
		\|_p \leq \tfrac{\beta}{\epsilon^{\frac \kappa 2 -1}} \big( \| \A^{\epsilon}_{r, t}\|_p + \|\B^{\epsilon}_{r, t}\|_p \big) \lesssim I_A(r,t)^{\frac 1 2} + I_B(r,t)^{\frac 1 2}.
	\end{equ}
	We can now apply Lemma~\ref{lem:bounds_IAB} as in the previous proof to obtain \eqref{eq:Xeps_cont_est_p}. The bound \eqref{eq:Xeps_cont_est_p_spatial} then follows in the same way as in Lemma~\ref{lem:EW_cont_est}.
\end{proof}

We can now use the estimates of Lemmas~\ref{lem:EW_cont_est} and~\ref{lem:X_cont_est} to control $\cU_t(\varphi_x^{(n)})$,  $X_t^{\epsilon, \varphi_x^{(n)}}$,  which allows us to conclude that both $\cU_t$ and $X_t^\epsilon$ belong to $\C^\gamma([0,T], \C^\alpha)$.

\begin{proposition}\label{prop:holder_Calpha+bound}
Let $d\geq 3$, $\gamma \in (0, \f 1 2)$ and $\alpha < 1 - \f \kappa 2 - 2\gamma$. 
\begin{enumerate}
\item\label{item1}	Then, there exists a version of $\cU$ in $\C^{\gamma} ([0,T], \C^{\alpha})$. 
\item  Let $\sigma$ be Lipschitz continuous, and let $\xi$ satisfies Assumption~\ref{assump-noise}.
\begin{enumerate}
\item\label{item2a} Then, there exist $ p=p_{\gamma,\alpha} >d$, such that for any $\beta~<~\beta_0(p_{\gamma,\alpha})$ the following holds
	\begin{equ} \label{eq:tight_moments_bound}
		\sup_{\epsilon \in (0,1)} \norm{ \| X^\epsilon \|_{\C^\gamma([0,T], \C^\alpha)} }_p < \infty. 
	\end{equ}
	\item\label{item2b}  For sufficiently small $\beta < \beta_{\gamma,\alpha}$,   $\{  X^\epsilon\}_{\epsilon \in (0,1)}$ is tight in $\C^{\gamma}([0,T], \C^{\alpha})$. 
\end{enumerate}
\end{enumerate}
\end{proposition}

\begin{proof}
	Let us denote with $\zeta_t$ either $\cU_t$ or $X^\epsilon_t$. We want to verify that $\zeta\in \C_{p, T}^{\gamma_0, \alpha_0}$ and apply Proposition \ref{prop:time-holder_distribution_process-app}.
	Recalling $\varphi^{(n)}=2^{nd} \varphi(2^n\cdot)$ and 
	applying  Lemmas~\ref{lem:EW_cont_est}-\ref{lem:X_cont_est}
	with test functions $\varphi^{(n)}_x$. For any $\gamma_0 \in (0, \f 1 2)$ and $\alpha_0 < 1 - \f \kappa 2 - 2\gamma_0$, we obtain the estimates
	\begin{equs}
		\sup_{x \in  \R^d} 
		\norm{\langle \zeta_{r,t}, \varphi^{(n)}_x \rangle }_{p} &\lesssim 
		2^{-n\alpha_0} |t -r|^{\gamma_0}
		\| \varphi \|_\infty, 
		\\
		\sup_{x\neq y \in  \R^d} 
		\norm{\langle \zeta_{r,t}, \varphi^{(n)}_x - \varphi^{(n)}_y \rangle }_{p} &\lesssim
		2^{-n(\alpha_0-1)} |t -r|^{\gamma_0} | \varphi |_\Lip.
	\end{equs}
	Hence \eqref{eq:lattice+time_control}--\eqref{eq:spatial_control} are satisfied
	and both $\cU_t$ and $X^\epsilon_t$ belong to $\C_{p, T}^{\gamma_0, \alpha_0}$ with  $\|X^\epsilon\|_{\C_{p, T}^{\gamma_0, \alpha_0}}$ uniformly bounded in $\epsilon$.  
	Choosing $p$ sufficiently large and $\beta < \beta_0(p)$ small enough so that 
	Lemma~\ref{lem:X_cont_est} holds, we can then apply Proposition \ref{prop:time-holder_distribution_process-app}
	which yields \eqref{eq:tight_moments_bound} as claimed.
	
We proceed to prove the last statement (\ref{item2b}). Given $\gamma$ and $\alpha$, we can pick real numbers $\gamma'$ and $\alpha'$ such that 
	$$ \gamma < \gamma' < \tfrac  1 2, \qquad \alpha < \alpha' < 1 - \tfrac \kappa 2 - 2\gamma'.$$  
	Since $\C^{\alpha'}$ is compactly embedded
	in $\C^\alpha$, $\C^{\gamma'}([0,T], \C^{\alpha'})$ is compactly embedded in $\C^{\gamma}([0,T], \C^{\alpha})$. 
	Thus $K_\Lambda := \{ X \in  \C^{\gamma}([0,T], \C^{\alpha}) \,: \,  \| X \|_{\C^{\gamma'}([0,T], \C^{\alpha'}) } \leq \Lambda\}$ is  compact in $\C^{\gamma}([0,T], \C^{\alpha})$, for any $\Lambda >0$. Given the uniform moment 
	bound \eqref{eq:tight_moments_bound}, which holds for sufficiently small $\beta < \beta_0(p_{\gamma', \alpha'})$,  we have
	\begin{equ}
		\mathbb{P} ( X^\epsilon \notin K_\Lambda) = \mathbb{P} ( \| X^\epsilon \|_{\C^{\gamma'}([0,T], \C^{\alpha'})} > \Lambda ) \leq \frac{\E \big[ \, \| X^\epsilon \|_{\C^{\gamma'}([0,T], \C^{\alpha'})}^p \,\big]}{\Lambda^p} \lesssim \frac 1 {\Lambda^p}.
	\end{equ}
	Hence $ X^\epsilon$ is tight in $\C^{\gamma}([0,T], \C^{\alpha})$. This completes the proof of the proposition.
\end{proof}

\begin{acks}[Acknowledgments]
This research was supported by the Royal Society through MH's research professorship
RP\textbackslash R1\textbackslash 191065,  by the 
EPSRC through XML's  grant EP/V026100/1, and by the CDT grant EP/S023925/1 where the latter is 
co-investigator. LG is supported by the EPSRC Centre for Doctoral Training in Mathematics of 
Random Systems: Analysis, Modelling 
and Simulation (EP/S023925/1).

  The authors would like to thank the anonymous referees for their
 helpful suggestions.
\end{acks}

\bibliographystyle{imsart-nameyear} 
\bibliography{bibliography}       

\end{document}